\documentclass[oneside,english]{amsart}
\usepackage[T1]{fontenc}
\usepackage[latin9]{inputenc}
\usepackage{amsthm}
\usepackage{amstext}
\usepackage{amssymb}
\usepackage{cite}
\usepackage{esint}
\usepackage{float}
\usepackage{graphicx}
\usepackage{multirow}
\usepackage{subfig}
\floatstyle{plaintop}
\restylefloat{table}

\makeatletter

\numberwithin{equation}{section}
\numberwithin{figure}{section}
\theoremstyle{plain}
\newtheorem{thm}{\protect\theoremname}[section]
  \theoremstyle{definition}
  \newtheorem{defn}[thm]{\protect\definitionname}
  \theoremstyle{plain}
  \newtheorem{prop}[thm]{\protect\propositionname}
  \theoremstyle{remark}
  \newtheorem{rem}{\protect\remarkname}[section]
  \theoremstyle{plain}
  
  \theoremstyle{plain}
  \newtheorem{cor}[thm]{\protect\corollaryname}
  \theoremstyle{plain}
  \newtheorem{lem}[thm]{\protect\lemmaname}
  \theoremstyle{conjecture}

\makeatletter
\@addtoreset{thm}{section}
\makeatother

\makeatother

\usepackage{babel}
  \providecommand{\corollaryname}{Corollary}
  \providecommand{\definitionname}{Definition}
  \providecommand{\factname}{Fact}
  \providecommand{\lemmaname}{Lemma}
  \providecommand{\propositionname}{Proposition}
  \providecommand{\remarkname}{Remark}
  \providecommand{\conjecturename}{Conjecture}
\providecommand{\theoremname}{Theorem}

\begin{document}

\title[Stability for $2$-D and $3$-D Linear Systems Based on Curvature]
{Description of Stability for Two and\\ Three-Dimensional Linear Time-Invariant Systems\\ Based on Curvature and Torsion}

\author[Y. Wang]{Yuxin Wang}
\author[H. Sun]{Huafei Sun}
\author[Y. Song]{Yang Song}
\author[Y. Cao]{Yueqi Cao}
\author[S. Zhang]{Shiqiang Zhang}

\thanks{This subject is supported by the National Natural Science Foundations of China (No. 61179031.)}
\thanks{School of Mathematics and Statistics, Beijing Institute of Technology, Beijing 100081, P.~R.~China}
\thanks{E-mail: wangyuxin@bit.edu.cn huafeisun@bit.edu.cn frank230316@126.com 1120143609@bit.edu.cn 1120141935@bit.edu.cn}
\thanks{Huafei Sun is the corresponding author}

\begin{abstract}
This paper focuses on using curvature and torsion to describe the stability of linear time-invariant system. We prove that for a two-dimensional system $\dot{r}(t)= Ar(t)$, (i) if there exists an initial value, such that zero is not the limit of curvature of trajectory as $t\to+\infty$, then the zero solution of the system is stable; (ii) if there exists an initial value, such that the limit of curvature of trajectory is infinity as $t\to+\infty$, then the zero solution of the system is asymptotically stable. For a three-dimensional system, (i) if there exists a measurable set whose Lebesgue measure is greater than zero, such that for all initial values in this set, zero is not the limit of curvature of trajectory as $t\to+\infty$, then the zero solution of the system is stable; (ii) if the coefficient matrix is invertible, and there exists a measurable set whose Lebesgue measure is greater than zero, such that for all initial values in this set, the limit of curvature of trajectory is infinity as $t\to+\infty$, then the zero solution of the system is asymptotically stable; (iii) if there exists a measurable set whose Lebesgue measure is greater than zero, such that for all initial values in this set, zero is not the limit of torsion of trajectory as $t\to+\infty$, then the zero solution of the system is asymptotically stable.
\end{abstract}



\maketitle

\section{Introduction}

Stability is an important subject in the control theory, which is the premise for the control system to work properly. In 1892, Russian mathematician Lyapunov (\!\!\cite{Lyapunov}) gave a rigorous mathematical definition and research method for stability of motion, which laid the foundation for stability theory.
Linear systems are the most basic objects in the field of control science. Linear system theory is the basis of many other branches of system control theory. The stability of linear systems has some well-known criteria, such as the method of Lyapunov functions, Routh-Hurwitz criterion, Mikhailov criterion, and Nyquist criterion.

It is well known that the general relativity theory by Einstein used differential geometry where the curvatures describing how curved of the space play the significant role. As S.~S.~Chern said, curvature is the core concept of differential geometry. Curvatures are used in lots of research fields, such as \cite{Yao1,Yao2}. In the differential geometry of three-dimensional Euclidean space, curvature and torsion describe the degree of bending and twisting of the curve, respectively. Curvature and torsion are invariants of rigid motion. Curvature function $\kappa(t)>0$ and torsion function $\tau(t)$ determine the only curve in the space $\mathbb{R}^3$ (\!\!\cite{Carmo}).

The aim of this paper is to use curvature and torsion to describe the stability of the zero solution of linear time-invariant system $\dot{r}(t)=Ar(t)$. Our main results are as follows.\\

\begin{thm} \label{thm 2-dim}
Suppose that $\dot{r}(t)=Ar(t)$ is a linear time-invariant system, where $A$ is a $2\times2$ real matrix, $r(t)\in\mathbb{R}^2$, and $\dot{r}(t)$ is the derivative of $r(t)$.
Denote by $\kappa(t)$ curvature of trajectory of a solution $r(t)$. \\
$(1)$ If there exists an initial value $r(0)\in\mathbb{R}^2$, such that $\lim\limits_{t\to+\infty}\kappa(t)\neq0$ or $\lim\limits_{t\to+\infty}\kappa(t)$ does not exist,
then the zero solution of the system is stable; \\
$(2)$ if there exists an initial value $r(0)\in\mathbb{R}^2$, such that $\lim\limits_{t\to+\infty}\kappa(t)=+\infty$, then the zero solution of the system is asymptotically stable. \\
\end{thm}

\begin{thm} \label{thm1}
 Suppose that $\dot{r}(t)=Ar(t)$ is a linear time-invariant system, where $A$ is a $3\times3$ real matrix, $r(t)\in\mathbb{R}^3$, and $\dot{r}(t)$ is the derivative of $r(t)$.
Let $\kappa(t)$ and $\tau(t)$ be curvature and torsion of trajectory of a solution $r(t)$, respectively. \\
$(1)$ If there exists a measurable set $E_1\subseteq \mathbb{R}^3$ whose Lebesgue measure is greater than $0$, such that for all $r(0)\in E_1$, $\lim\limits_{t\to+\infty}\kappa(t)\neq0$ or $\lim\limits_{t\to+\infty}\kappa(t)$ does not exist, then the zero solution of the system is stable; \\
$(2)$ if $A$ is invertible, and there exists a measurable set $E_2\subseteq \mathbb{R}^3$ whose Lebesgue measure is greater than $0$, such that for all $r(0)\in E_2$, $\lim\limits_{t\to+\infty}\kappa(t)=+\infty$, then the zero solution of the system is asymptotically stable; \\
$(3)$ if there exists a measurable set $E_3\subseteq \mathbb{R}^3$ whose Lebesgue measure is greater than $0$, such that for all $r(0)\in E_3$, $\lim\limits_{t\to+\infty}\tau(t)\neq0$ or $\lim\limits_{t\to+\infty}\tau(t)$ does not exist, then the zero solution of the system is asymptotically stable.
\end{thm}

The paper is organized as follows.
In Section 2, we review some basic concepts and propositions.
In Section 3, we give the relationship between curvatures of trajectories of two equivalent systems.
In Section 4 and Section 5 we prove Theorem \ref{thm 2-dim} and Theorem \ref{thm1}, respectively.
Several examples are given in Section 6. Finally, Section 7 concludes the paper.

\section{Preliminaries}
In this paper, $a\times b$ denotes the vector product of $a$ and $b$, and $(a, b, c)=(a\times b)\cdot c$ denotes the scalar triple product of $a$, $b$ and $c$, for any $a, b, c\in\mathbb{R}^3$. The norm $\|x\|$ denotes the Euclidean norm of $x=(x_1,x_2,\cdots,x_n)^\mathrm{T}\in\mathbb{R}^n$, namely, $\|x\|=\sqrt{\sum_{i=1}^n x_i^2}$. The determinant of matrix $A$ is denoted by $\mathrm{det}\, A$. The eigenvalues of matrix $A$ are denoted by $\lambda_i(A)\,(i=1,2,\cdots,n)$.

The following concepts and results can be found in \cite{Carmo,Chen,Horn,Lyapunov,Marsden,Perko,Strang}.

\begin{defn}[\!\!\cite{Carmo}]
Let $r:[0, +\infty)\to\mathbb{R}^3$ be a smooth curve. The functions
\begin{align*}
\kappa(t)=\frac{\left\|\dot{r}(t)\times\ddot{r}(t)\right\|}{\left\|\dot{r}(t)\right\|^3}
\end{align*}
and
\begin{align*}
\tau(t)=\frac{\left(\dot{r}(t),\ddot{r}(t),\dddot{r}(t)\right)}{\left\|\dot{r}(t)\times\ddot{r}(t)\right\|^2}
\end{align*}
are called curvature and torsion of curve $r(t)$, respectively.
\end{defn}

\begin{defn}[\!\!\cite{Perko}]
The system of ordinary differential equations
\begin{align}\label{system1}
\dot{r}(t)=Ar(t)
\end{align}
is called a linear time-invariant system, where $A$ is an $n\times n$ real constant matrix, $r(t)\in\mathbb{R}^n$, and $\dot{r}(t)$ is the derivative of $r(t)$.
\end{defn}

\begin{prop}[\!\!\cite{Perko}]\label{ODE}
Let $A$ be an $n\times n$ real matrix. Then for a given $r_0\in\mathbb{R}^n$, the initial value problem
\begin{align}
\left\{
\begin{aligned}\label{system}
\dot{r}(t)&=Ar(t),\\
r(0)&=r_0
\end{aligned}
\right.
\end{align}
has a unique solution given by
\begin{align*}
r(t)=\mathrm{e}^{tA}r_0.
\end{align*}
\par
The curve $r(t)$ is called the trajectory of system (\ref{system}) with the initial value $r_0\in\mathbb{R}^n$.
\end{prop}

\begin{defn}[\!\!\cite{Chen,Marsden}]
The solution $r(t)\equiv0$ of the differential equations $\dot{r}(t)=Ar(t)$ is called the zero solution of the linear time-invariant system.
If for every $\varepsilon >0$,
there exists a $\delta=\delta(\varepsilon)>0$,
such that $\|r(0)\|<\delta $ implies that $\|r(t)\|<\varepsilon, \,\forall t\geqslant0$,
where $r(t)=\mathrm{e}^{tA}r(0)$ is a solution of the equations $\dot{r}(t)=Ar(t)$,
and $r(0)$ is the initial value of $r(t)$,
then we say that the zero solution of $\dot{r}(t)=Ar(t)$ is stable.
If the zero solution is not stable, then we say that it is unstable.
\par
Suppose that the zero solution of system $\dot{r}(t)=Ar(t)$ is stable,
and for every solution $r(t)=\mathrm{e}^{tA}r(0)$,
there exists a $\delta_1\,(0<\delta_1\leqslant\delta)$,
such that $\|r(0)\|<\delta_1$ implies that $\lim\limits_{t\to+\infty}r(t)=0$,
then we say that the zero solution of $\dot{r}(t)=Ar(t)$ is asymptotically stable.
\end{defn}

\begin{prop}[\!\!\cite{Chen}]\label{asy.stable}
The zero solution of system (\ref{system1}) is stable if and only if all eigenvalues of matrix $A$ have nonpositive real parts, namely,
\begin{align*}
\mathrm{Re}\{\lambda_i(A)\}\leqslant 0\quad (i=1,2,\cdots,n),
\end{align*}
and the eigenvalues with zero real parts correspond only to the simple elementary factors of matrix $A$.
\par
The zero solution of system (\ref{system1}) is asymptotically stable if and only if every eigenvalues of matrix $A$ have a negative real part, namely,
\begin{align*}
\mathrm{Re}\{\lambda_i(A)\}<0\quad (i=1,2,\cdots,n).
\end{align*}
\end{prop}

\begin{prop}[\!\!\cite{Chen}]
Suppose that $A$ and $B$ are two $n\times n$ real matrices, and there exists an $n\times n$ real invertible matrix $P$, such that $A=P^{-1}BP$. For system (\ref{system1}), let $v(t)=Pr(t)$. Then the system after the transformation is
\begin{align}\label{system2}
\dot{v}(t)=Bv(t).
\end{align}
System (\ref{system2}) is said to be equivalent to system (\ref{system1}), and $v(t)=Pr(t)$ is called an equivalence transformation.
\end{prop}

\begin{prop}[\!\!\cite{Chen}]
Let $A$ and $B$ be two $n\times n$ real matrices, and $A$ is similar to $B$. Then the zero solution of the system $\dot{r}(t)=Ar(t)$ is (asymptotically) stable if and only if the zero solution of the system $\dot{v}(t)=Bv(t)$ is (asymptotically) stable.
\end{prop}

\begin{prop}[Real Jordan canonical form \cite{Horn}]\label{Jordan1}
Let $A$ be an $n\times n$ real matrix. Then $A$ is similar to a block diagonal real matrix
\begin{align*}
\begin{pmatrix}
 \begin{matrix}
  C_{n_1}(a_1,b_1)&&\\
  &C_{n_2}(a_2,b_2)&\\
  &&\ddots
 \end{matrix}
 &&\text{\LARGE$0$}\\
 &C_{n_p}(a_p,b_p)&\\
 \text{\LARGE$0$}&&
 \begin{matrix}
  J_{n_{p+1}}(\lambda_{p+1})&&\\
  &\ddots&\\
  &&J_{n_{r}}(\lambda_r)
 \end{matrix}
\end{pmatrix},
\end{align*}
where\\
$(1)$ for $k\in\{1,2,\cdots,p\}$, $\lambda_k=a_k+\sqrt{-1} b_k$ and $\bar\lambda_k=a_k-\sqrt{-1} b_k$
\ $(a_k,b_k\in\mathbb{R},\text{and}\ b_k>0)$ are eigenvalues, and
\begin{align*}
C_{n_k}(a_k,b_k)=
\begin{pmatrix}
\Lambda_k & I_2&&&\\[0.6em]
&\Lambda_k & I_2&&\\
&&\Lambda_k &\ddots&\\
&&&\ddots & I_2\\[0.5em]
&&&&\Lambda_k
\end{pmatrix}_{2n_k\times 2n_k},
\end{align*}
where
$\Lambda_k=
\begin{pmatrix}
 a_k&b_k\\
-b_k&a_k
\end{pmatrix},
I_2=
\begin{pmatrix}
1&0\\
0&1
\end{pmatrix};$\\\\
$(2)$ for $j\in\{p+1,p+2,\cdots,r\}$, $\lambda_j\in\mathbb{R}$ is a real eigenvalue, and
\begin{align*}
J_{n_j}(\lambda_j)=
\begin{pmatrix}
\lambda_j &1&&&\\[0.6em]
&\lambda_j &1&&\\
&&\lambda_j &\ddots&\\
&&&\ddots &1\\[0.5em]
&&&&\lambda_j
\end{pmatrix}_{n_j\times n_j}.
\end{align*}
\end{prop}

\section{Relationship Between the Curvatures of Two Equivalent Systems}\label{SectionRelationship}
In this section, we give the relationship between curvatures of trajectories of two equivalent systems. In fact, we have the following theorems.
\begin{thm}\label{thm k1}
Suppose that three-dimensional system $\dot{r}(t)=Ar(t)$ is equivalent to system $\dot{v}(t)=Bv(t)$, where
 $A=P^{-1}BP$, and $v(t)=Pr(t)$ is the equivalence transformation. Let $\kappa_r(t)$ and $\kappa_v(t)$ be curvatures of trajectories $r(t)$ and $v(t)$, respectively. Then we have
\begin{align*}
\lim\limits_{t\to+\infty}\kappa_v(t)=0 \iff &\lim\limits_{t\to+\infty}\kappa_r(t)=0, \nonumber\\
\lim\limits_{t\to+\infty}\kappa_v(t)=+\infty \iff &\lim\limits_{t\to+\infty}\kappa_r(t)=+\infty, \nonumber\\
\exists C_1, C_2>0, \ \exists T>0, \ \mathrm{s.t.,}&\ \kappa_v(t)\in\left[C_1, C_2\right], \forall t>T \nonumber\\
\iff \exists \tilde{C_1}, \tilde{C_2}>0, \ \exists \tilde{T}>0, \ \mathrm{s.t.,}&\ \kappa_r(t)\in\left[\tilde{C_1}, \tilde{C_2}\right], \forall t>\tilde{T}.
\end{align*}
\end{thm}

\begin{thm}\label{thm k2}
Under the assumptions of Theorem \ref{thm k1}, let $\tau_r(t)$ and $\tau_v(t)$ be torsions of trajectories $r(t)$ and $v(t)$, respectively. Then we have
\begin{align*}
\lim\limits_{t\to+\infty}\tau_v(t)=0 \iff &\lim\limits_{t\to+\infty}\tau_r(t)=0, \nonumber\\
\lim\limits_{t\to+\infty}\tau_v(t)=\infty \iff &\lim\limits_{t\to+\infty}\tau_r(t)=\infty, \nonumber\\
\exists C_1, C_2>0, \ \exists T>0,& \ \mathrm{s.t.,}\ \left|\tau_v(t)\right|\in\left[C_1, C_2\right], \forall t>T \nonumber\\
\iff \exists \tilde{C_1}, \tilde{C_2}>0, \ \exists \tilde{T}>0,& \ \mathrm{s.t.,}\ \left|\tau_r(t)\right|\in\left[\tilde{C_1}, \tilde{C_2}\right], \forall t>\tilde{T}.
\end{align*}
\end{thm}

The following theorem is an analogue of Theorem \ref{thm k1} for two-dimensional case.

\begin{thm}\label{thm k1 2-dim}
Suppose that two-dimensional system $\dot{r}(t)=Ar(t)$ is equivalent to system $\dot{v}(t)=Bv(t)$, where
 $A=P^{-1}BP$, and $v(t)=Pr(t)$ is the equivalence transformation. Let $\kappa_r(t)$ and $\kappa_v(t)$ be curvatures of trajectories $r(t)$ and $v(t)$, respectively. Then we have
\begin{align*}
\lim\limits_{t\to+\infty}\kappa_v(t)=0 \iff &\lim\limits_{t\to+\infty}\kappa_r(t)=0, \nonumber\\
\lim\limits_{t\to+\infty}\kappa_v(t)=+\infty \iff &\lim\limits_{t\to+\infty}\kappa_r(t)=+\infty, \nonumber\\
\exists C_1, C_2>0, \ \exists T>0, \ \mathrm{s.t.},&\ \kappa_v(t)\in\left[C_1, C_2\right], \forall t>T \nonumber\\
\iff \exists \tilde{C_1}, \tilde{C_2}>0, \ \exists \tilde{T}>0, \ \mathrm{s.t.},&\ \kappa_r(t)\in\left[\tilde{C_1}, \tilde{C_2}\right], \forall t>\tilde{T}.
\end{align*}
\end{thm}

The proofs of Theorem \ref{thm k1} and \ref{thm k2} are shown in Appendix \ref{Proof}, and Theorem \ref{thm k1 2-dim} can be proved in a similar way.

Theorem \ref{thm k1}, \ref{thm k2} and \ref{thm k1 2-dim} give the relationship between curvatures (or torsions) of trajectories of two equivalent systems.

\section{Two-Dimensional Systems}\label{Section 2-dim}

In the case of two-dimensional systems, Proposition \ref{Jordan1} becomes the following result.

\begin{prop}\label{Jordan2-2dim}
Let $A$ be a $2\times2$ real matrix. Then the matrix $A$ is similar to one of the following three cases:
\begin{align}
(1)
\begin{pmatrix}\lambda_1&0\\0&\lambda_2\end{pmatrix}\
(\lambda_1, \lambda_2\in\mathbb{R}),
\qquad
(2)
\begin{pmatrix}a&b\\-b&a\end{pmatrix}\
(a, b\in\mathbb{R},\ b>0),
\qquad
(3)
\begin{pmatrix}\lambda&1\\0&\lambda\end{pmatrix}\
(\lambda\in\mathbb{R}).
\nonumber
\end{align}
\end{prop}

Note that we have Theorem \ref{thm k1 2-dim} in Section \ref{SectionRelationship}. In order to prove Theorem \ref{thm 2-dim}, we need only consider the three cases in Proposition \ref{Jordan2-2dim}.
\par
We notice that the initial value $r(0)=(x_0,y_0)^\mathrm{T}$ may affect curvature $\kappa(t)$ of curve $r(t)$.
For simplicity, in the calculations below, we always assume that $x_0y_0\neq 0$. It will be seen later (Section \ref{2-dim result}) that this assumption does not affect the correctness of the proof of Theorem \ref{thm 2-dim}.
\par
We give the calculation details of curvatures in Appendix \ref{cal 2-dim}.

\subsection{Case 1} \label{section 2-dim (1)}

\

First, we consider the case of the coefficient matrix is a real diagonal matrix, namely,
\begin{align*}
A=\begin{pmatrix}\lambda_1&0\\0&\lambda_2\end{pmatrix}\
(\lambda_1, \lambda_2\in\mathbb{R}).
\end{align*}
If $\lambda_1=\lambda_2=0$, then $\kappa(t)\equiv0$;
if $\lambda_1^2+\lambda_2^2 \neq 0$, then the square of curvature $\kappa(t)$ of curve $r(t)$ is
\begin{align*}
\kappa^2(t)
=\frac{\left\{\lambda_1 \lambda_2(\lambda_2-\lambda_1)x_0y_0\right\}^2 \cdot \mathrm{e}^{2(\lambda_1+\lambda_2)t}}
{\left\{(\lambda_1x_0)^2 \cdot \mathrm{e}^{2\lambda_1t}+(\lambda_2y_0)^2 \cdot \mathrm{e}^{2\lambda_2t}\right\}^3}.
\end{align*}

Through the analysis of  all situations of the eigenvalues, we have the following Table \ref{2-dim table 1}. Without loss of generality, we suppose that $\lambda_1\geqslant\lambda_2$.

\begin{table}[htbp]
\caption{Two-Dimensional Systems for Case 1}
$
\begin{array}{|c|c|c|c|c|c|}
\hline
 &\multirow{2}*{\text{Eigenvalues}} & \multirow{2}*{det$A$}
& {\text{Asymptotically}}  & \multirow{2}*{\text{Stable}} & \text{Curvature}\ \kappa(t) \\
 &&& {\text{Stable}} &  &  (t\to+\infty) \\ \hline
(1)& 0<\lambda_2<\lambda_1 &\neq0& {\text{No}} & {\text{No}} & \kappa(t)\to 0\\ \hline
(2)& 0<\lambda_2=\lambda_1 &\neq0& {\text{No}} & {\text{No}} & \kappa(t)\equiv0\\ \hline
(3)& 0=\lambda_2<\lambda_1 &0& {\text{No}} & {\text{No}} & \kappa(t)\equiv0\\ \hline
(4)& \lambda_2<0<\lambda_1 &\neq0& {\text{No}} & {\text{No}} & \kappa(t)\to 0\\ \hline
(5)& 0=\lambda_2=\lambda_1 &0& {\text{No}} & {\text{Yes}} & \kappa(t)\equiv 0\\ \hline
(6)& \lambda_2<0=\lambda_1 &0& {\text{No}} & {\text{Yes}} & \kappa(t)\equiv0\\ \hline
(7)& \lambda_2=\lambda_1<0 &\neq0& {\text{Yes}} & {\text{Yes}} & \kappa(t)\equiv 0\\ \hline
(8)& \lambda_2<\lambda_1<0 &\neq0& {\text{Yes}} & {\text{Yes}} &
\kappa(t)\to\left\{
\begin{aligned}
&0, &2\lambda_1>\lambda_2,\\
&C, &2\lambda_1=\lambda_2,\\
&+\infty, &2\lambda_1<\lambda_2.
\end{aligned}\right.
\\ \hline
\end{array}
$
\label{2-dim table 1}
\end{table}

\par In Table \ref{2-dim table 1}, the column of ``Asymptotically Stable'' presents the asymptotic stability of zero solution, the column of ``Stable'' presents the stability of zero solution, and $C$ denotes a positive constant, whose value depends on the initial value $r(0)$.
\par We see that if $\lim\limits_{t\to+\infty}\kappa(t)=+\infty$ or $C\ (0<C<+\infty)$, then the zero solution of the system is asymptotically stable.

\begin{rem}
If $\mathrm{det}\, A=0$, then the number $0$ is an eigenvalue of matrix $A$, by Proposition \ref{asy.stable}, the zero solution of the system is not asymptotically stable.
\end{rem}

\begin{rem}
We know that the eigenvalues of matrix $A$ correspond only to the Jordan blocks of $1\times1$ when $A$ is a diagonal matrix. By Proposition \ref{asy.stable}, if $A$ is a diagonal matrix, then the zero solution of system (\ref{system1}) is stable if and only if
$
\mathrm{Re}\{\lambda_i(A)\}\leqslant 0 ~ (i=1,2,\cdots,n).
$
\end{rem}

\begin{rem}
By the expression of curvature, for any given coefficient matrix $A$ of Case $1$, if for some initial value $r(0)\in\mathbb{R}^2$ which satisfies $x_0y_0\neq0$, we have $\lim\limits_{t\to+\infty}\kappa(t)=0$ (or +$\infty$, or a constant $C>0$, respectively), then for an arbitrary $r(0)\in\mathbb{R}^2$ satisfying $x_0y_0\neq0$, we still have $\lim\limits_{t\to+\infty}\kappa(t)=0$ (or +$\infty$, or a constant $\tilde{C}>0$, respectively). In fact, we have Theorem \ref{2-dim initial} in Section \ref{2-dim result}.
\end{rem}
By using the similar ways, we can give the results for Case $2$ and $3$.
\subsection{Case 2}

\

In the case of
\begin{align*}
A=\begin{pmatrix}a&b\\-b&a\end{pmatrix}\
(a, b\in\mathbb{R},\ b>0),
\end{align*}
the square of curvature $\kappa(t)$ of curve $r(t)$ is
\begin{align*}
\kappa^2(t)
=\frac{b^2}{(a^2+b^2)(x_0^2+y_0^2)\cdot \mathrm{e}^{2at}}.
\end{align*}

Through the analysis of all situations of the eigenvalues, we have the following Table \ref{2-dim table 2}.
\begin{table}[htbp]
\caption{Two-Dimensional Systems for Case 2}
$
\begin{array}{|c|c|c|c|c|c|}
\hline
 &\multirow{2}*{\text{Eigenvalues}} & \multirow{2}*{det$A$}
& {\text{Asymptotically}}  & \multirow{2}*{\text{Stable}} & \text{Curvature}\ \kappa(t) \\
 &&& {\text{Stable}} &  &  (t\to+\infty) \\ \hline
(1)& a>0 &\neq0& {\text{No}} & {\text{No}} & \kappa(t)\to 0\\ \hline
(2)& a=0 &\neq0& {\text{No}} & {\text{Yes}} & \kappa(t)\equiv C\\ \hline
(3)& a<0 &\neq0& {\text{Yes}} & {\text{Yes}} & \kappa(t)\to+\infty\\ \hline
\end{array}
$
\label{2-dim table 2}
\end{table}

\par In Table \ref{2-dim table 2}, $C$ denotes a positive constant, whose value depends on the initial value $r(0)$.
We see that if $\lim\limits_{t\to+\infty}\kappa(t)=+\infty$ or $C\ (0<C<+\infty)$, then the zero solution of the system is stable. In particular, if $\lim\limits_{t\to+\infty}\kappa(t)=+\infty$, then the zero solution of the system is asymptotically stable.

\subsection{Case 3}

\

In the case of
\begin{align*}
A=\begin{pmatrix}\lambda&1\\0&\lambda\end{pmatrix}\
(\lambda\in\mathbb{R}),
\end{align*}
the square of curvature $\kappa(t)$ of curve $r(t)$ is
\begin{align*}
\kappa^2(t)
=\frac{ \lambda^4 y_0^4 \cdot \mathrm{e}^{4\lambda t}}
{\left\{(\lambda x_0+\lambda y_0 \cdot t+y_0)^2+(\lambda y_0)^2\right\}^3 \cdot \mathrm{e}^{6\lambda t} }
=\frac{ \lambda^4 y_0^4 }
{g(t) \cdot \mathrm{e}^{2\lambda t} },
\end{align*}
where $g(t)$ is a polynomial in $t$. If $\lambda\neq0$, then $g(t)= \left( \lambda^6 y_0^6 \right) \cdot t^{6} + \sum_{i=0}^{5} a_i t^i$ is a polynomial of degree 6 in $t$; if $\lambda=0$, then $g(t)=y_0^6$ is a constant.

Through the analysis of all situations of the eigenvalues, we have the following Table \ref{2-dim table 3}.

\begin{table}[htbp]
\caption{Two-Dimensional Systems for Case 3}
$
\begin{array}{|c|c|c|c|c|c|}
\hline
 &\multirow{2}*{\text{Eigenvalues}} & \multirow{2}*{det$A$}
& {\text{Asymptotically}}  & \multirow{2}*{\text{Stable}} & \text{Curvature}\ \kappa(t) \\
 &&& {\text{Stable}} &  &  (t\to+\infty) \\ \hline
(1)& \lambda>0 &\neq0& {\text{No}} & {\text{No}} & \kappa(t)\to 0\\ \hline
(2)& \lambda=0 &0& {\text{No}} & {\text{No}} & \kappa(t)\equiv 0\\ \hline
(3)& \lambda<0 &\neq0& {\text{Yes}} & {\text{Yes}} & \kappa(t)\to+\infty\\ \hline
\end{array}
$
\label{2-dim table 3}
\end{table}

We find that if $\lim\limits_{t\to+\infty}\kappa(t)\neq0$, then the zero solution of the system is asymptotically stable.

\subsection{Results}\label{2-dim result}

\

We give the results in this subsection.
\par
We notice that if the initial value $r(0)=(x_0,y_0)^\mathrm{T}$ satisfies $x_0y_0=0$, then we have $\lim\limits_{t\to+\infty}\kappa(t)=0$. Hence we can obtain the following proposition by Table \ref{2-dim table 1}, \ref{2-dim table 2}, and \ref{2-dim table 3}.

\begin{prop}
Under the assumptions of Theorem \ref{thm 2-dim}, additionally assuming that $A$ is a matrix in real Jordan canonical form. \\
$(1)$ If there exists an initial value $r(0)\in\mathbb{R}^2$, such that $\lim\limits_{t\to+\infty}\kappa(t)\neq0$ or $\lim\limits_{t\to+\infty}\kappa(t)$ does not exist,
then the zero solution of the system is stable; \\
$(2)$ if there exists an initial value $r(0)\in\mathbb{R}^2$, such that $\lim\limits_{t\to+\infty}\kappa(t)=+\infty$, then the zero solution of the system is asymptotically stable.
\end{prop}

Combined with Theorem \ref{thm k1 2-dim}, we complete the proof of Theorem \ref{thm 2-dim}.
\par
\begin{rem}
If all eigenvalues of $A$ are real numbers, namely, matrix $A$ is similar to one of the two cases $(1)$ and $(3)$ in Proposition \ref{Jordan2-2dim}. According to Table \ref{2-dim table 1} and \ref{2-dim table 3}, combined with Theorem \ref{thm k1 2-dim}, we obtain the following proposition.
\end{rem}

\begin{prop}
Under the assumptions of Theorem \ref{thm 2-dim}, together with the assumption that all eigenvalues of $A$ are real numbers. If there exists an initial value $r(0)\in\mathbb{R}^2$, such that$\lim\limits_{t\to+\infty}\kappa(t)\neq0$ or $\lim\limits_{t\to+\infty}\kappa(t)$ does not exist,
then the zero solution of the system is asymptotically stable.
\end{prop}
\par
\begin{rem}
By the expressions of curvature in the above cases, and Table \ref{2-dim table 1}, \ref{2-dim table 2}, \ref{2-dim table 3}, we obtain the following theorem.
\end{rem}

\begin{thm}\label{2-dim initial}
Under the assumptions of Theorem \ref{thm 2-dim}, additionally assuming that $A$ is a matrix in real Jordan canonical form. For any given initial value $r(0)\in\mathbb{R}^2$, curvature $\kappa(t)$ of trajectory $r(t)$ of system (\ref{system}) is subject to one of the following three cases:
\begin{align*}
\lim\limits_{t\to+\infty}\kappa(t)=0, \qquad
\lim\limits_{t\to+\infty}\kappa(t)=+\infty, \qquad
\lim\limits_{t\to+\infty}\kappa(t)=C\ (0<C<+\infty).
\end{align*}
Moreover, if for some initial value $r(0)\in\mathbb{R}^2$ which satisfies $x_0y_0\neq0$, we have $\lim\limits_{t\to+\infty}\kappa(t)=0$ (or $+\infty$, or a constant $C>0$, respectively), then for an arbitrary $r(0)\in\mathbb{R}^2$ satisfying $x_0y_0\neq0$, we still have $\lim\limits_{t\to+\infty}\kappa(t)=0$ (or $+\infty$, or a constant $\tilde{C}>0$, respectively).
\end{thm}

\section{Three-Dimensional Systems} \label{SectionJordan}\setcounter{table}{0}

In the case of three-dimensional systems, Proposition \ref{Jordan1} becomes the following result.

\begin{prop}\label{Jordan2}
Let $A$ be a $3\times3$ real matrix. Then the matrix $A$ is similar to one of the following four cases:
\begin{align}
&(1)
\begin{pmatrix}\lambda_1&0&0\\0&\lambda_2&0\\0&0&\lambda_3\end{pmatrix}\
(\lambda_1, \lambda_2, \lambda_3\in\mathbb{R}),
\qquad
(2)
\begin{pmatrix}a&b&0\\-b&a&0\\0&0&\lambda_3\end{pmatrix}\
(a, b, \lambda_3\in\mathbb{R},\ b>0),
\nonumber\\
&(3)
\begin{pmatrix}\lambda_1&1&0\\0&\lambda_1&0\\0&0&\lambda_2\end{pmatrix}\
(\lambda_1, \lambda_2\in\mathbb{R}),
\phantom{,\lambda_3}\qquad
(4)
\begin{pmatrix}\lambda&1&0\\0&\lambda&1\\0&0&\lambda\end{pmatrix}\
(\lambda\in\mathbb{R}).
\nonumber
\end{align}
\end{prop}

Note that we have Theorem \ref{thm k1} and \ref{thm k2} in Section \ref{SectionRelationship}. In order to prove Theorem \ref{thm1}, we need only consider the four cases in Proposition \ref{Jordan2}.
\par
We notice that the initial value $r(0)=(x_0,y_0,z_0)^\mathrm{T}$ may affect curvature $\kappa(t)$ and torsion $\tau(t)$ of curve $r(t)$.
For simplicity, in the calculations below, we always assume that $x_0y_0z_0\neq 0$. It will be seen later (Section \ref{section xyz=0} and \ref{3-dim result}) that this assumption does not affect the correctness of the proof of Theorem \ref{thm1}.
\par
We give the calculation details of curvatures and torsions in Appendix \ref{cal 3-dim}.

\subsection{Case 1} \label{section 3-dim (1)}

\

First, we consider the case of the coefficient matrix is a real diagonal matrix, namely,
\begin{align*}
A=\begin{pmatrix}\lambda_1&0&0\\0&\lambda_2&0\\0&0&\lambda_3\end{pmatrix}\
(\lambda_1, \lambda_2, \lambda_3\in\mathbb{R}).
\end{align*}
If $\lambda_1=\lambda_2=\lambda_3=0$, then $\kappa(t)\equiv0$ and $\tau(t)\equiv0$;
if $\lambda_1^2+\lambda_2^2+\lambda_3^2 \neq 0$, then the square of curvature $\kappa(t)$ of curve $r(t)$ is
\begin{align*}
\kappa^2(t)
=& \left\{
\left[\lambda_2 \lambda_3 (\lambda_3-\lambda_2) y_0 z_0\right]^2 \cdot \mathrm{e}^{2(\lambda_2+\lambda_3)t}
+\left[\lambda_1 \lambda_3 (\lambda_1-\lambda_3) x_0 z_0\right]^2 \cdot \mathrm{e}^{2(\lambda_1+\lambda_3)t}
\right.\nonumber\\& \left.
+\left[\lambda_1 \lambda_2 (\lambda_2-\lambda_1) x_0 y_0\right]^2 \cdot \mathrm{e}^{2(\lambda_1+\lambda_2)t}
\right\}\left.\middle/
{\left\{(\lambda_1 x_0)^2 \mathrm{e}^{2\lambda_1t}
+(\lambda_2 y_0)^2 \mathrm{e}^{2\lambda_2t}
+(\lambda_3 z_0)^2 \mathrm{e}^{2\lambda_3t}\right\}^3}
.\right.
\end{align*}
If $[\lambda_2 \lambda_3 (\lambda_3-\lambda_2) ]^2
+[\lambda_1 \lambda_3 (\lambda_1-\lambda_3) ]^2
+[\lambda_1 \lambda_2 (\lambda_2-\lambda_1) ]^2=0$,
then $\left\|\dot{r}(t)\times\ddot{r}(t)\right\|\equiv0$ and we have $\tau(t)\equiv0$;
if $\left\|\dot{r}(t)\times\ddot{r}(t)\right\|\neq0$, then torsion of curve $r(t)$ is
\begin{align*}
\tau(t)
=& \left.
{\lambda_1\lambda_2\lambda_3(\lambda_2-\lambda_1)(\lambda_3-\lambda_1)(\lambda_3-\lambda_2) x_0y_0z_0 \cdot
\mathrm{e}^{(\lambda_1+\lambda_2+\lambda_3)t}}
 \middle/
\Big\{\left[\lambda_2 \lambda_3 (\lambda_3-\lambda_2) y_0z_0\right]^2 \cdot \mathrm{e}^{2(\lambda_2+\lambda_3)t}
\right. \nonumber\\
&+\left[\lambda_1 \lambda_3 (\lambda_1-\lambda_3) x_0z_0\right]^2 \cdot \mathrm{e}^{2(\lambda_1+\lambda_3)t}
+\left[\lambda_1 \lambda_2 (\lambda_2-\lambda_1) x_0y_0\right]^2 \cdot \mathrm{e}^{2(\lambda_1+\lambda_2)t}
\Big\}. \nonumber\\
\end{align*}

Through the analysis of all situations of the eigenvalues, we have the following Table \ref{table 1}.
Without loss of generality, we suppose that $\lambda_1\geqslant\lambda_2\geqslant\lambda_3$.

\begin{table}[htbp]
\caption{Three-Dimensional Systems for Case 1}
$
\begin{footnotesize}
\begin{array}{|c|c|c|c|c|c|c|}
\hline
&\multirow{2}*{\text{Eigenvalues}} & \multirow{2}*{det$A$}
& {\text{Asymptotically}}  & \multirow{2}*{\text{Stable}} & \text{Curvature}\ \kappa(t) & \text{Torsion}\ \tau(t)\\
 &&& {\text{Stable}} &  &  (t\to+\infty) &  (t\to+\infty)\\ \hline
(1)& 0<\lambda_3<\lambda_2<\lambda_1 &\neq0& {\text{No}} & {\text{No}} & \kappa(t)\to 0 &  \tau(t)\to 0\\ \hline
(2)& 0<\lambda_3<\lambda_2=\lambda_1 &\neq0& {\text{No}} & {\text{No}} & \kappa(t)\to 0 &  \tau(t)\equiv 0\\ \hline
(3)& 0<\lambda_3=\lambda_2<\lambda_1 &\neq0& {\text{No}} & {\text{No}} & \kappa(t)\to 0 &  \tau(t)\equiv 0\\ \hline
(4)& 0<\lambda_3=\lambda_2=\lambda_1 &\neq0& {\text{No}} & {\text{No}} & \kappa(t)\equiv 0 &  \tau(t)\equiv 0\\ \hline
(5)& 0=\lambda_3<\lambda_2<\lambda_1 &0& {\text{No}} & {\text{No}} & \kappa(t)\to 0 &  \tau(t)\equiv 0\\ \hline
(6)& 0=\lambda_3<\lambda_2=\lambda_1 &0& {\text{No}} & {\text{No}} & \kappa(t)\equiv 0 &  \tau(t)\equiv 0\\ \hline
(7)& \lambda_3<0<\lambda_2<\lambda_1 &\neq0& {\text{No}} & {\text{No}} & \kappa(t)\to 0 &  \tau(t)\to 0\\ \hline
(8)& \lambda_3<0<\lambda_2=\lambda_1 &\neq0& {\text{No}} & {\text{No}} & \kappa(t)\to 0 &  \tau(t)\equiv 0\\ \hline
(9)& 0=\lambda_3=\lambda_2<\lambda_1 &0& {\text{No}} & {\text{No}} & \kappa(t)\equiv 0 &  \tau(t)\equiv 0\\ \hline
(10)& \lambda_3<0=\lambda_2<\lambda_1 &0& {\text{No}} & {\text{No}} & \kappa(t)\to 0 &  \tau(t)\equiv 0\\ \hline
(11)& \lambda_3<\lambda_2<0<\lambda_1 &\neq0& {\text{No}} & {\text{No}} & \kappa(t)\to 0 &  \tau(t)\to 0\\ \hline
(12)& \lambda_3=\lambda_2<0<\lambda_1 &\neq0& {\text{No}} & {\text{No}} & \kappa(t)\to 0 &  \tau(t)\equiv 0\\ \hline
(13)& 0=\lambda_3=\lambda_2=\lambda_1 &0& {\text{No}} & {\text{Yes}} & \kappa(t)\equiv 0 &  \tau(t)\equiv 0\\ \hline
(14)& \lambda_3<0=\lambda_2=\lambda_1 &0& {\text{No}} & {\text{Yes}} & \kappa(t)\equiv 0 &  \tau(t)\equiv 0\\ \hline
(15)& \lambda_3<\lambda_2<0=\lambda_1 &0& {\text{No}} & {\text{Yes}} &
\kappa(t)\to\left\{
\begin{aligned}
&0, &2\lambda_2>\lambda_3\\
&C, &2\lambda_2=\lambda_3\\
&+\infty, &2\lambda_2<\lambda_3
\end{aligned}\right.
&  \tau(t)\equiv 0\\ \hline
(16)& \lambda_3=\lambda_2<0=\lambda_1 &0& {\text{No}} & {\text{Yes}} & \kappa(t)\equiv 0 &  \tau(t)\equiv 0\\ \hline
(17)& \lambda_3<\lambda_2<\lambda_1<0 &\neq0& {\text{Yes}} & {\text{Yes}} &
\kappa(t)\to\left\{
\begin{aligned}
&0, &2\lambda_1>\lambda_2\\
&C, &2\lambda_1=\lambda_2\\
&+\infty, &2\lambda_1<\lambda_2
\end{aligned}\right.
 &
\tau(t)\to\left\{
\begin{aligned}
&0, &\lambda_1+\lambda_2-\lambda_3>0\\
&C, &\lambda_1+\lambda_2-\lambda_3=0\\
&\infty, &\lambda_1+\lambda_2-\lambda_3<0
\end{aligned}\right.\\ \hline
(18)& \lambda_3<\lambda_2=\lambda_1<0 &\neq0& {\text{Yes}} & {\text{Yes}} &
\kappa(t)\to\left\{
\begin{aligned}
&0, &2\lambda_1>\lambda_3\\
&C, &2\lambda_1=\lambda_3\\
&+\infty, &2\lambda_1<\lambda_3
\end{aligned}\right.
 &  \tau(t)\equiv 0\\ \hline
(19)& \lambda_3=\lambda_2<\lambda_1<0 &\neq0& {\text{Yes}} & {\text{Yes}} &
\kappa(t)\to\left\{
\begin{aligned}
&0, &2\lambda_1>\lambda_2\\
&C, &2\lambda_1=\lambda_2\\
&+\infty, &2\lambda_1<\lambda_2
\end{aligned}\right.
 &  \tau(t)\equiv 0\\ \hline
(20)& \lambda_3=\lambda_2=\lambda_1<0 &\neq0& {\text{Yes}} & {\text{Yes}} & \kappa(t)\equiv 0 &  \tau(t)\equiv 0\\ \hline
\end{array}
\end{footnotesize}
$
\label{table 1}
\end{table}

\par In Table \ref{table 1}, $C$ in the column of torsion denotes a non-zero constant, and each $C$ in the column of curvature denotes a positive constant, respectively. The value of each $C$ depends on the initial value $r(0)$.
\par
We see that if $\lim\limits_{t\to+\infty}\kappa(t)=+\infty$ or $C\ (0<C<+\infty)$, then the zero solution of the system is stable;
if $\mathrm{det}\, A\neq0$ and $\lim\limits_{t\to+\infty}\kappa(t)=+\infty$ or $C\ (0<C<+\infty)$, then the zero solution of the system is asymptotically stable;
if $\lim\limits_{t\to+\infty}\tau(t)=\infty$ or $C\ (C\neq 0)$, then the zero solution of the system is asymptotically stable.
\begin{rem}
By the expression of curvature, for any given coefficient matrix $A$ of Case $1$, if for some initial value $r(0)\in\mathbb{R}^3$ that satisfies $x_0y_0z_0\neq0$, we have $\lim\limits_{t\to+\infty}\kappa(t)=0$ (or $\infty$, or a constant $C>0$, respectively), then for an arbitrary $r(0)\in\mathbb{R}^3$ that satisfies $x_0y_0z_0\neq0$, we still have $\lim\limits_{t\to+\infty}\kappa(t)=0$ (or $\infty$, or a constant $\tilde{C}>0$, respectively). There is a similar result for torsion $\tau(t)$. In fact, we have Theorem \ref{3-dim initial} and Corollary \ref{3-dim initial cor} in Section \ref{3-dim result}.
\end{rem}
By using the similar ways, we can give the results for Case $2$, $3$ and $4$.

\subsection{Case 2}

\

In the case of
\begin{align*}
A=
\begin{pmatrix}a&b&0\\-b&a&0\\0&0&\lambda_3\end{pmatrix}\
(a, b, \lambda_3\in\mathbb{R},\ b>0),
\end{align*}
the square of curvature $\kappa(t)$ of curve $r(t)$ is
\begin{align*}
\kappa^2(t)
= \frac{ (a^2 + b^2)(x_0^2 + y_0^2) \cdot
  \{ \lambda_3^2 [(a-\lambda_3)^2+b^2]z_0^2 \cdot \mathrm{e}^{2(a+\lambda_3)t}
  + b^2(a^2 + b^2)(x_0^2 + y_0^2) \cdot \mathrm{e}^{4at} \} }
 { \{(a^2 + b^2)(x_0^2 + y_0^2) \cdot \mathrm{e}^{2at}
 +(\lambda_3 z_0)^2 \cdot \mathrm{e}^{2\lambda_3 t} \}^3 },
\end{align*}
and torsion of curve $r(t)$ is
\begin{align*}
\tau(t)
=\frac{ -b\lambda_3\{(a-\lambda_3)^2+b^2\} z_0 }
{ \lambda_3^2 \{(a-\lambda_3)^2+b^2\} z_0^2 \cdot \mathrm{e}^{\lambda_3 t}
  + b^2(a^2 + b^2)(x_0^2+y_0^2) \cdot \mathrm{e}^{(2a-\lambda_3)t} }.
\end{align*}

Through the analysis of all situations of the eigenvalues, we have the following Table \ref{table 2}.

\begin{table}[htbp]
\caption{Three-Dimensional Systems for Case 2}
$
\begin{footnotesize}
\begin{array}{|c|c|c|c|c|c|c|}
\hline
&\multirow{2}*{\text{Eigenvalues}} & \multirow{2}*{det$A$}
& {\text{Asymptotically}}  & \multirow{2}*{\text{Stable}} & \text{Curvature}\ \kappa(t) & \text{Torsion}\ \tau(t)\\
 &&& {\text{Stable}} &  &  (t\to+\infty) &  (t\to+\infty)\\ \hline
(1)& \lambda_3>0,\ a>0 &\neq0& {\text{No}} & {\text{No}} & \kappa(t)\to 0 & \tau(t)\to 0\\ \hline
(2)& \lambda_3>0,\ a=0 &\neq0& {\text{No}} & {\text{No}} & \kappa(t)\to 0 & \tau(t)\to 0\\ \hline
(3)& \lambda_3>0,\ a<0 &\neq0& {\text{No}} & {\text{No}} & \kappa(t)\to 0 & \tau(t)\to 0\\ \hline
(4)& \lambda_3=0,\ a>0  &0& {\text{No}} & {\text{No}} & \kappa(t)\to 0 & \tau(t)\equiv 0\\ \hline
(5)& \lambda_3=0,\ a=0  &0& {\text{No}} & {\text{Yes}} & \kappa(t)\equiv C & \tau(t)\equiv 0\\ \hline
(6)& \lambda_3=0,\ a<0  &0& {\text{No}} & {\text{Yes}} & \kappa(t)\to +\infty & \tau(t)\equiv 0\\ \hline
(7)& \lambda_3<0,\ a>0 &\neq0& {\text{No}} & {\text{No}} & \kappa(t)\to 0 & \tau(t)\to 0\\ \hline
(8)& \lambda_3<0,\ a=0 &\neq0& {\text{No}} & {\text{Yes}} & \kappa(t)\to C & \tau(t)\to 0\\ \hline
(9)& \lambda_3<0,\ a<0 &\neq0& {\text{Yes}} & {\text{Yes}} &
\kappa(t)\left\{
\begin{aligned}
&\to 0, &2\lambda_3>a,\\
&\equiv C, &2\lambda_3=a,\\
&\to +\infty, &2\lambda_3<a.
\end{aligned}\right.
&
\tau(t)\to\left\{
\begin{aligned}
&0, &2a>\lambda_3,\\
&C, &2a=\lambda_3,\\
&\infty, &2a<\lambda_3.
\end{aligned}\right.
\\ \hline
\end{array}
\end{footnotesize}
$
\label{table 2}
\end{table}

\par In Table \ref{table 2}, $C$ in the column of torsion denotes a non-zero constant, and each $C$ in the column of curvature denotes a positive constant, respectively. The value of each $C$ depends on the initial value $r(0)$.
\par
We see that if $\lim\limits_{t\to+\infty}\kappa(t)=+\infty$ or $C\ (0<C<+\infty)$, then the zero solution of the system is stable;
if $\mathrm{det}\, A\neq0$ and $\lim\limits_{t\to+\infty}\kappa(t)=+\infty$, then the zero solution of the system is asymptotically stable;
if $\lim\limits_{t\to+\infty}\tau(t)=\infty$ or $C\ (C\neq 0)$, then the zero solution of the system is asymptotically stable.

\subsection{Case 3}

\

In the case of
\begin{align*}
A=\begin{pmatrix}\lambda_1&1&0\\0&\lambda_1&0\\0&0&\lambda_2\end{pmatrix}\
(\lambda_1, \lambda_2\in\mathbb{R}),
\end{align*}
the square of curvature $\kappa(t)$ of curve $r(t)$ is
\begin{align*}
\kappa^2(t)
=&\Big\{
\big[\lambda_1 \lambda_2 (\lambda_2-\lambda_1) y_0 z_0 \big]^2 \cdot \mathrm{e}^{2(\lambda_1+\lambda_2)t}
\nonumber
\\&
+(\lambda_2 z_0)^2 \big[\lambda_1(\lambda_1 -\lambda_2 )y_0 \cdot t + (\lambda_1^2 x_0 - \lambda_1\lambda_2 x_0 + 2\lambda_1 y_0 -\lambda_2 y_0) \big] ^2 \cdot \mathrm{e}^{2(\lambda_1+\lambda_2)t}
\nonumber
\\&
+(\lambda_1^2 y_0^2)^2 \cdot \mathrm{e}^{4\lambda_1t}
\Big\}
\left.\middle/
{\Big\{ \big[ \lambda_1 y_0 \cdot t + ( \lambda_1 x_0 + y_0 ) \big]^2 \cdot \mathrm{e}^{2\lambda_1t}
+(\lambda_1 y_0)^2 \cdot \mathrm{e}^{2\lambda_1t}
+(\lambda_2 z_0)^2 \cdot \mathrm{e}^{2\lambda_2t}\Big\}^3}
\right..
\end{align*}

If $\lambda_1=\lambda_2=0$, then $\left\|\dot{r}(t)\times\ddot{r}(t)\right\|\equiv0$ and every trajectory $r(t)$ is a straight line, therefore we have $\tau(t)\equiv0$.
If $\lambda_1^2+\lambda_2^2 \neq 0$, then torsion of curve $r(t)$ is
\begin{align*}
\tau(t)
=&\left. { -\lambda_1^2 \lambda_2 (\lambda_1-\lambda_2)^2 y_0^2 z_0 \cdot \mathrm{e}^{(2\lambda_1+\lambda_2)t} }
\middle/\right.
\Big\{
\big[\lambda_1 \lambda_2 (\lambda_2-\lambda_1) y_0 z_0 \big]^2 \cdot \mathrm{e}^{2(\lambda_1+\lambda_2)t}
\nonumber
\\&
+(\lambda_2 z_0)^2 \big[\lambda_1(\lambda_1 -\lambda_2 )y_0 \cdot t + (\lambda_1^2 x_0 - \lambda_1\lambda_2 x_0 + 2\lambda_1 y_0 -\lambda_2 y_0) \big] ^2 \cdot \mathrm{e}^{2(\lambda_1+\lambda_2)t}
\nonumber
\\&
+(\lambda_1^2 y_0^2)^2 \cdot \mathrm{e}^{4\lambda_1t}
\Big\}.
\end{align*}

Through the analysis of all situations of the eigenvalues, we have the following Table \ref{table 3}.

\begin{table}[htbp]
\caption{Three-Dimensional Systems for Case 3}
$
\begin{footnotesize}
\begin{array}{|c|c|c|c|c|c|c|}
\hline
 &\multirow{2}*{\text{Eigenvalues}} & \multirow{2}*{det$A$}
& {\text{Asymptotically}}  & \multirow{2}*{\text{Stable}} & \text{Curvature}\ \kappa(t) & \text{Torsion}\ \tau(t)\\
 &&& {\text{Stable}} &  &  (t\to+\infty) &  (t\to+\infty)\\ \hline
(1)& \lambda_1, \lambda_2>0, \lambda_1\neq\lambda_2 &\neq0& {\text{No}} &{\text{No}} & \kappa(t)\to 0 & \tau(t)\to 0\\ \hline
(2)& \lambda_1=\lambda_2>0 &\neq0& {\text{No}} &{\text{No}} & \kappa(t)\to 0 & \tau(t)\equiv 0\\ \hline
(3)& \lambda_1=\lambda_2=0 &0& {\text{No}} & {\text{No}} & \kappa(t)\equiv 0 & \tau(t)\equiv 0\\ \hline
(4)& \lambda_1=0,\ \lambda_2\neq0 &0& {\text{No}} & {\text{No}} & \kappa(t)\to 0 & \tau(t)\equiv 0\\ \hline
(5)& \lambda_2=0,\ \lambda_1>0 &0& {\text{No}} & {\text{No}} & \kappa(t)\to 0 & \tau(t)\equiv 0\\ \hline
(6)& \lambda_2=0,\ \lambda_1<0 &0& {\text{No}} & {\text{Yes}} & \kappa(t)\to +\infty & \tau(t)\equiv 0\\ \hline
(7)& \lambda_1\lambda_2<0 &\neq0& {\text{No}} & {\text{No}} & \kappa(t)\to 0 & \tau(t)\to 0\\ \hline
(8)& \lambda_1=\lambda_2<0 &\neq0& {\text{Yes}} & {\text{Yes}} & \kappa(t)\to +\infty & \tau(t)\equiv 0\\ \hline
(9)& \lambda_1, \lambda_2<0,\ \lambda_1\neq\lambda_2 &\neq0& {\text{Yes}} & {\text{Yes}} &
\kappa(t)\to\left\{
\begin{aligned}
&0, &2\lambda_2\geqslant\lambda_1,\\
&+\infty, &2\lambda_2<\lambda_1.
\end{aligned}\right.
&
\tau(t)\to\left\{
\begin{aligned}
&0, &2\lambda_1>\lambda_2,\\
&C, &2\lambda_1=\lambda_2,\\
&\infty, &2\lambda_1<\lambda_2.
\end{aligned}\right.
\\ \hline
\end{array}
\end{footnotesize}
$
\label{table 3}
\end{table}

\par In Table \ref{table 3}, $C$ denotes a non-zero constant, whose value depends on the initial value $r(0)$. We see that if $\lim\limits_{t\to+\infty}\kappa(t)\neq0$, then the zero solution of the system is stable;
if $\mathrm{det}\, A\neq0$ and $\lim\limits_{t\to+\infty}\kappa(t)\neq0$, then the zero solution of the system is asymptotically stable;
if $\lim\limits_{t\to+\infty}\tau(t)=\infty$ or $C(C\neq0)$, then the zero solution of the system is asymptotically stable.

\subsection{Case 4}

\

In the case of
\begin{align*}
A=\begin{pmatrix}\lambda&1&0\\0&\lambda&1\\0&0&\lambda\end{pmatrix}\
(\lambda\in\mathbb{R}).
\end{align*}
the square of curvature $\kappa(t)$, and torsion $\tau(t)$ of curve $r(t)$ are
\begin{align*}
\kappa^2(t)
=\frac{ f(t) }{ g(t) \cdot \mathrm{e}^{2\lambda t} },\qquad\qquad
\tau (t)
=\frac{ -\lambda^3 z_0^3 }{ f(t) \cdot \mathrm{e}^{\lambda t} },
\end{align*}
where $f(t)$ and $g(t)$ are polynomials in $t$.
If $\lambda\neq0$, then
$f(t)= \left( \frac14 \lambda^4 z_0^4 \right) \cdot t^4 + \sum_{i=0}^3 a_i t^i$ is a quartic polynomial in $t$,
and $g(t)= \left( \frac{1}{64} \lambda^6 z_0^6 \right) \cdot t^{12} + \sum_{i=0}^{11} b_i t^i$ is a polynomial of degree 12 in $t$;
if $\lambda=0$, then $f(t)=z_0^4$ is a constant, and $g(t)=z_0^6 t^6 + \sum_{i=0}^{5} b_i t^i$ is a polynomial of degree $6$ in $t$.
\par
Hence we have the following Table \ref{table 4}.

\begin{table}[htbp]
\caption{Three-Dimensional Systems for Case 4}
$
\begin{array}{|c|c|c|c|c|c|c|}
\hline
&\multirow{2}*{\text{Eigenvalues}} & \multirow{2}*{det$A$}
& {\text{Asymptotically}}  & \multirow{2}*{\text{Stable}} & \text{Curvature}\ \kappa(t) & \text{Torsion}\ \tau(t)\\
 &&& {\text{Stable}} &  &  (t\to+\infty) &  (t\to+\infty)\\ \hline
(1)& \lambda>0 &\neq0& {\text{No}} & {\text{No}} & \kappa(t)\to 0 & \tau(t)\to 0\\ \hline
(2)& \lambda=0 &0& {\text{No}} & {\text{No}} & \kappa(t)\to 0 & \tau(t)\equiv 0\\ \hline
(3)& \lambda<0 &\neq0& {\text{Yes}} & {\text{Yes}} & \kappa(t)\to+\infty & \tau(t)\to \infty\\ \hline
\end{array}
$
\label{table 4}
\end{table}

We see that if $\lim\limits_{t\to+\infty}\kappa(t)\neq0$ or $\lim\limits_{t\to+\infty}\tau(t)\neq0$, then the zero solution of the system is asymptotically stable.

\subsection{The Situation of $x_0y_0z_0=0$}\label{section xyz=0}

\

In the calculations above, we always assume that initial value $r(0)=(x_0,y_0,z_0)^\mathrm{T}$ satisfies $x_0y_0z_0\neq 0$. On the other hand, in the situation of $x_0y_0z_0=0$, the condition that zero is not the limit of the curvature $\kappa(t)$ (or torsion $\tau(t)$) of trajectory as $t\to +\infty$ does not ensure the stability of the zero solution of the system; the condition that $A$ is invertible and $\lim\limits_{t\to+\infty}\kappa(t)=+\infty$ does not ensure the asymptotic stability of the zero solution of the system.
\par
Nevertheless, the Lebesgue measure of the set $\{ r(0)=(x_0,y_0,z_0)^\mathrm{T}|x_0y_0z_0=0 \} \subseteq \mathbb{R}^3$ is zero. Thus the situation of $x_0y_0z_0=0$ does not affect the correctness of Theorem \ref{thm1} in the case of $A$ is a matrix in real Jordan canonical form.

\subsection{Results}\label{3-dim result}

\

We give the results in this subsection.

By Table \ref{table 1}, \ref{table 2}, \ref{table 3}, and \ref{table 4}, we have the following theorem.
\begin{thm}\label{thm Jordan}
Under the assumptions of Theorem \ref{thm1}, additionally assuming that $A$ is a matrix in real Jordan canonical form. For any given initial value $r(0)\in\{ (x_0,y_0,z_0)^\mathrm{T}\in\mathbb{R}^3|x_0y_0z_0\neq0 \}$, \\
$(1)$ if $\lim\limits_{t\to+\infty}\kappa(t)\neq0$ or $\lim\limits_{t\to+\infty}\kappa(t)$ does not exist,
then the zero solution of the system is stable; \\
$(2)$ if $A$ is invertible and $\lim\limits_{t\to+\infty}\kappa(t)=+\infty$, then the zero solution of the system is asymptotically stable; \\
$(3)$ if $\lim\limits_{t\to+\infty}\tau(t)\neq0$ or $\lim\limits_{t\to+\infty}\tau(t)$ does not exist,
then the zero solution of the system is asymptotically stable.
\end{thm}

Noting that the Lebesgue measure of the set $\{ r(0)=(x_0,y_0,z_0)^\mathrm{T}|x_0y_0z_0=0 \} \subseteq \mathbb{R}^3$ is zero, we have shown that Theorem \ref{thm1} holds if the coefficient matrix $A$ of system (\ref{system1}) is one of the four kinds of real Jordan canonical forms above.

Furthermore, combined with Theorem \ref{thm k1} and \ref{thm k2}, we have the following corollary.
\begin{cor}\label{thm v=Pr}
Suppose that $\dot{r}(t)=Ar(t)$ and $\dot{v}(t)=Jv(t)$ are two linear time-invariant systems, where $A$ and $J$ are two $3\times3$ real matrices, $J$ is a matrix in real Jordan canonical form,
$A=P^{-1}JP$, and $v(t)=Pr(t)$. Let $\kappa(t)$ and $\tau(t)$ be curvature and torsion of trajectory of a solution $r(t)$, respectively. For an arbitrary initial value $r(0)\in\mathbb{R}^3$ that satisfies each coordinate in the vector $Pr(0)$ is non-zero, \\
$(1)$ if $\lim\limits_{t\to+\infty}\kappa(t)\neq0$ or $\lim\limits_{t\to+\infty}\kappa(t)$ does not exist,
then the zero solution of the system is stable; \\
$(2)$ if $A$ is invertible and $\lim\limits_{t\to+\infty}\kappa(t)=+\infty$, then the zero solution of the system is asymptotically stable; \\
$(3)$ if $\lim\limits_{t\to+\infty}\tau(t)\neq0$ or $\lim\limits_{t\to+\infty}\tau(t)$ does not exist,
then the zero solution of the system is asymptotically stable.
\end{cor}
Noting that the Lebesgue measure of the set $\{ r(0)=P^{-1}v(0)|v(0)=(\tilde{x}_0, \tilde{y}_0, \tilde{z}_0)^\mathrm{T}, \mathrm{s.t.,}~\tilde{x}_0 \tilde{y}_0 \tilde{z}_0=0 \} \subseteq \mathbb{R}^3$ is zero, we complete the proof of Theorem \ref{thm1}.

\begin{rem}
If all eigenvalues of $A$ are real numbers, namely, matrix $A$ is similar to one of the three cases $(1)$, $(3)$, and $(4)$ in Proposition \ref{Jordan2}. According to Table \ref{table 1}, \ref{table 3}, and \ref{table 4}, combined with Theorem \ref{thm k1} and \ref{thm k2}, we obtain the following proposition.
\end{rem}

\begin{prop}\label{real eig}
Under the assumptions of Theorem \ref{thm1}, together with the assumption that $A$ is invertible, and all eigenvalues of $A$ are real numbers. If there exists a measurable set $E\subseteq \mathbb{R}^3$ whose Lebesgue measure is greater than $0$, such that for all $r(0)\in E$, $\lim\limits_{t\to+\infty}\kappa(t)\neq0$ or $\lim\limits_{t\to+\infty}\kappa(t)$ does not exist, then the zero solution of the system is asymptotically stable.
\end{prop}

\begin{rem}
By the expressions of curvature and torsion in the above cases, and Table \ref{table 1}, \ref{table 2}, \ref{table 3}, \ref{table 4}, we obtain the following theorem.
\end{rem}

\begin{thm}\label{3-dim initial}
Under the assumptions of Theorem \ref{thm1}, additionally assuming that $A$ is a matrix in real Jordan canonical form. For any given initial value $r(0)\in\mathbb{R}^3$, curvature $\kappa(t)$ of trajectory $r(t)$ of system (\ref{system}) is subject to one of the following three cases:
\begin{align*}
\lim\limits_{t\to+\infty}\kappa(t)=0, \qquad
\lim\limits_{t\to+\infty}\kappa(t)=+\infty, \qquad
\lim\limits_{t\to+\infty}\kappa(t)=C\ (0<C<+\infty);
\end{align*}
and torsion $\tau(t)$ of system (\ref{system}) is subject to one of the following three cases:
\begin{align*}
\lim\limits_{t\to+\infty}\tau(t)=0, \qquad
\lim\limits_{t\to+\infty}\tau(t)=\infty, \qquad \phantom{+}
\lim\limits_{t\to+\infty}\tau(t)=C\ (C\neq0). \phantom{<+\infty}
\end{align*}
Moreover, if for some initial value $r(0)\in\mathbb{R}^3$ which satisfies $x_0y_0z_0\neq0$, we have $\lim\limits_{t\to+\infty}\kappa(t)=0$ (or $+\infty$, or a constant $C>0$, respectively), then for an arbitrary $r(0)\in\mathbb{R}^3$ satisfying $x_0y_0z_0\neq0$, we still have $\lim\limits_{t\to+\infty}\kappa(t)=0$ (or $+\infty$, or a constant $\tilde{C}>0$, respectively). There is a similar result for torsion $\tau(t)$.
\end{thm}

Combined with Theorem \ref{thm k1} and \ref{thm k2}, we have the following corollary.

\begin{cor}\label{3-dim initial cor}
Under the assumptions of Corollary \ref{thm v=Pr}, for an arbitrary initial value $r(0)\in\mathbb{R}^3$, curvature $\kappa(t)$ of trajectory $r(t)$ of system (\ref{system}) is subject to one of the following three cases:
\begin{align*}
\lim\limits_{t\to+\infty}\kappa(t)=0, \qquad
\lim\limits_{t\to+\infty}\kappa(t)=+\infty, \qquad
\exists C_1, C_2>0, \ \exists T>0, \ \mathrm{s.t.,}&\ \kappa(t)\in\left[C_1, C_2\right], \forall t>T;
\end{align*}
and torsion $\tau(t)$ of system (\ref{system}) is subject to one of the following three cases:
\begin{align*}
\lim\limits_{t\to+\infty}\tau(t)=0, \qquad
\lim\limits_{t\to+\infty}\tau(t)=\infty, \qquad \phantom{+}
\exists \bar{C_1}, \bar{C_2}>0, \ \exists \bar{T}>0,& \ \mathrm{s.t.,}\ \left|\tau(t)\right|\in\left[\bar{C_1}, \bar{C_2}\right], \forall t>\bar{T}.
\end{align*}
Moreover, if for some initial value $r(0)\in\mathbb{R}^3$ that satisfies each coordinate in the vector $Pr(0)$ is non-zero, we have $\lim\limits_{t\to+\infty}\kappa(t)=0$ (or $+\infty$, or $\exists C_1, C_2>0, \ \exists T>0, \ \mathrm{s.t.,}\ \kappa(t)\in\left[C_1, C_2\right], \forall t>T$, respectively), then for an arbitrary $r(0)\in\mathbb{R}^3$ that satisfies each coordinate in the vector $Pr(0)$ is non-zero, we still have $\lim\limits_{t\to+\infty}\kappa(t)=0$ (or $+\infty$, or $\exists \tilde{C_1}, \tilde{C_2}>0, \ \exists \tilde{T}>0, \ \mathrm{s.t.,}\ \kappa(t)\in\left[\tilde{C_1}, \tilde{C_2}\right], \forall t>\tilde{T}$, respectively).
There is a similar result for torsion $\tau(t)$.
\end{cor}

\section{Examples}

In this section, we give several examples, which correspond to each case of Theorem \ref{thm 2-dim} and \ref{thm1}, respectively.

\subsection*{Example 1 (Theorem \ref{thm 2-dim}(1))}

\

Let
\begin{align*}
\begin{pmatrix}\dot x\\[0.5ex] \dot y\end{pmatrix}\
=\begin{pmatrix}18&25\\[0.5ex] -13&-18\end{pmatrix}\ \begin{pmatrix}x\\[0.5ex] y\end{pmatrix}\
\end{align*}
be a two-dimensional linear time-invariant system. Then we have
\begin{align*}
A=\begin{pmatrix}18&25\\[0.5ex] -13&-18\end{pmatrix}\ , \qquad
\mathrm{e}^{tA}
=\begin{pmatrix}\cos{t}+18\sin{t}&25\sin{t}
\\[1ex] -13\sin{t}&\cos{t}-18\sin{t}
\end{pmatrix},
\end{align*}
and $\dot r(t)=A\mathrm{e}^{tA}r(0)$ and $\ddot r(t)=A^2\mathrm{e}^{tA}r(0)$.
Let $r(0)=(1,1)^\mathrm{T}$ be the initial value of $r(t)$. Then the square of curvature $\kappa(t)$ of curve $r(t)$ is
\begin{align*}
\kappa^2(t)
&=\left( \frac{\dot x\ddot y-\ddot x\dot y}{ (\dot x^2+\dot y^2)^{3/2}} \right) ^2
=\frac{ 1369 }
{ 2 ( 703+702\cos{2t}-6\sin{2t} )^3 }.
\end{align*}
Therefore $\lim\limits_{t\to+\infty}\kappa(t)$ does not exist. Using Theorem \ref{thm 2-dim} (1), we obtain that the zero solution of the system is stable.

The trajectory of $r(t)$ and the graph of function $\kappa(t)$ are shown in Figure \ref{fig:1}.
\begin{figure}[h]
\centering
\subfloat[Trajectory]{
\label{fig:1(1)}
\begin{minipage}[h]{0.48\textwidth}
   \centering
   \includegraphics[angle=0,width=1\textwidth]{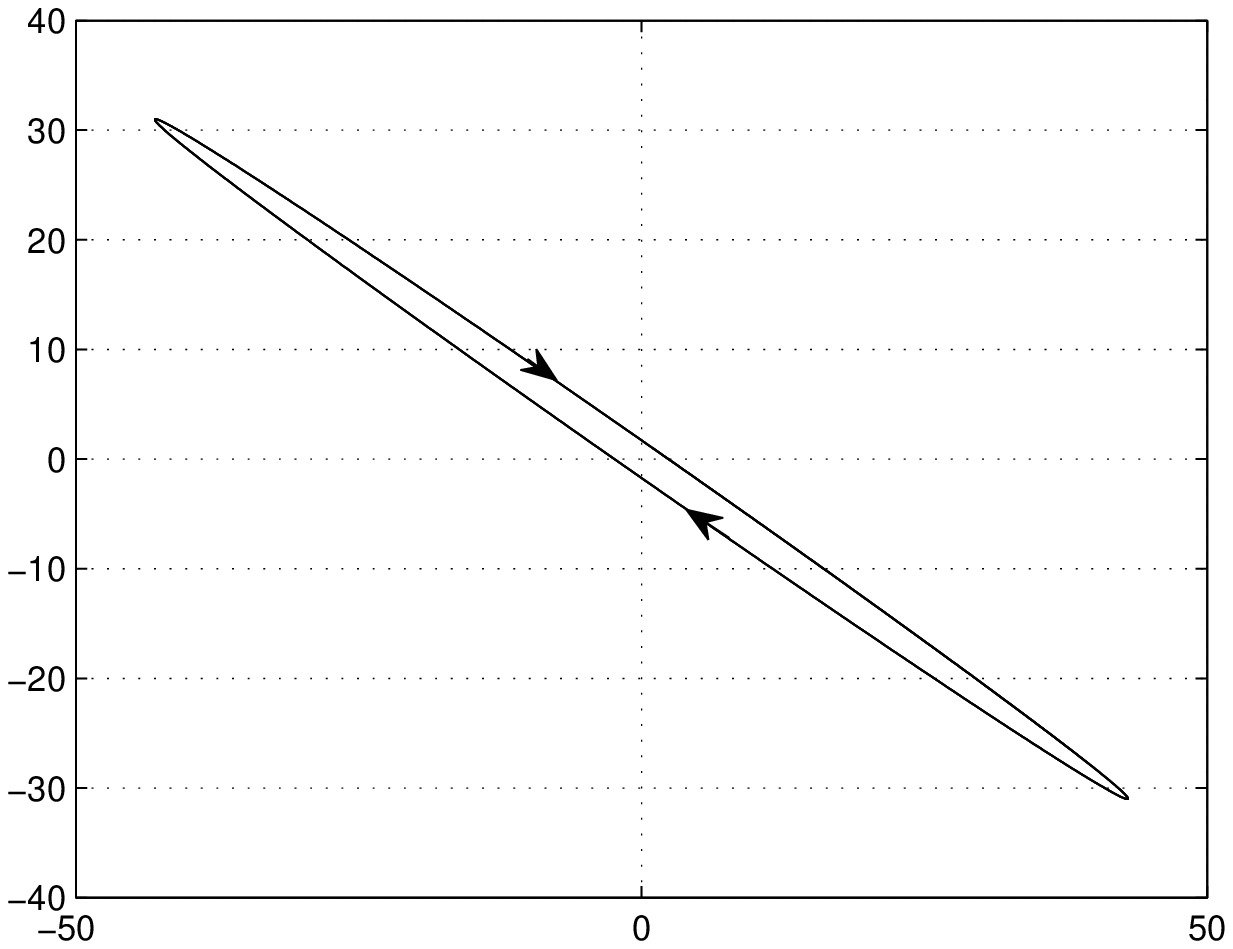}
\end{minipage}
}
\subfloat[Curvature]{
\label{fig:1(2)}
\begin{minipage}[h]{0.48\textwidth}
   \centering
   \includegraphics[angle=0,width=1\textwidth]{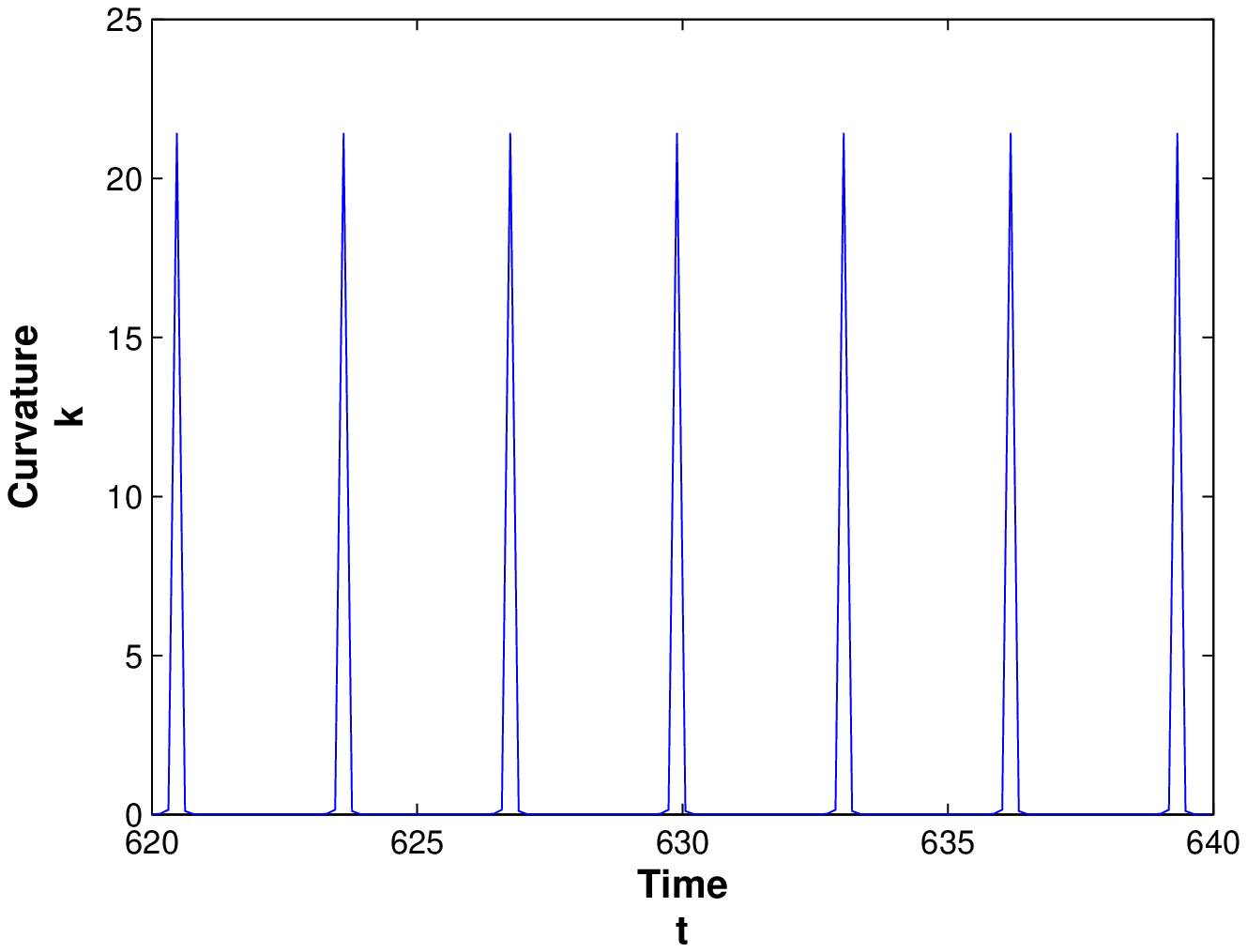}
\end{minipage}
}
\caption{Example 1}
\label{fig:1}
\end{figure}

\subsection*{Example 2 (Theorem \ref{thm 2-dim}(2))}

\

Let
\begin{align*}
\begin{pmatrix}\dot x\\[0.5ex] \dot y\end{pmatrix}\
=\begin{pmatrix}-4&-2\\[0.5ex] 1&-1\end{pmatrix}\ \begin{pmatrix}x\\[0.5ex] y\end{pmatrix}\
\end{align*}
be a two-dimensional linear time-invariant system. Then we have
\begin{align*}
A=\begin{pmatrix}-4&-2\\[0.5ex] 1&-1\end{pmatrix}\ , \qquad
\mathrm{e}^{tA}
=\mathrm{e}^{-3t}
\begin{pmatrix} 2-\mathrm{e}^{t} & 2(1-\mathrm{e}^{t})
\\[1ex] -1+\mathrm{e}^{t} & -1+2\mathrm{e}^{t}
\end{pmatrix}.
\end{align*}
Let $r(0)=(1,1)^\mathrm{T}$ be the initial value of $r(t)$.
Then the square of curvature $\kappa(t)$ of curve $r(t)$ is
\begin{align*}
\kappa^2(t)
=\frac{ \mathrm{e}^{8t} }{ 36(5-6\mathrm{e}^{t}+2\mathrm{e}^{2t})^3 }.
\end{align*}
Hence we have $\lim\limits_{t\to+\infty}\kappa(t)=+\infty$. Using Theorem \ref{thm 2-dim} (2), we obtain that the zero solution of the system is asymptotically stable.

The trajectory of $r(t)$ and the graph of function $\kappa(t)$ are shown in Figure \ref{fig:2}.
\begin{figure}[h]
\centering
\subfloat[Trajectory]{
\label{fig:2(1)}
\begin{minipage}[h]{0.48\textwidth}
   \centering
   \includegraphics[angle=0,width=1\textwidth]{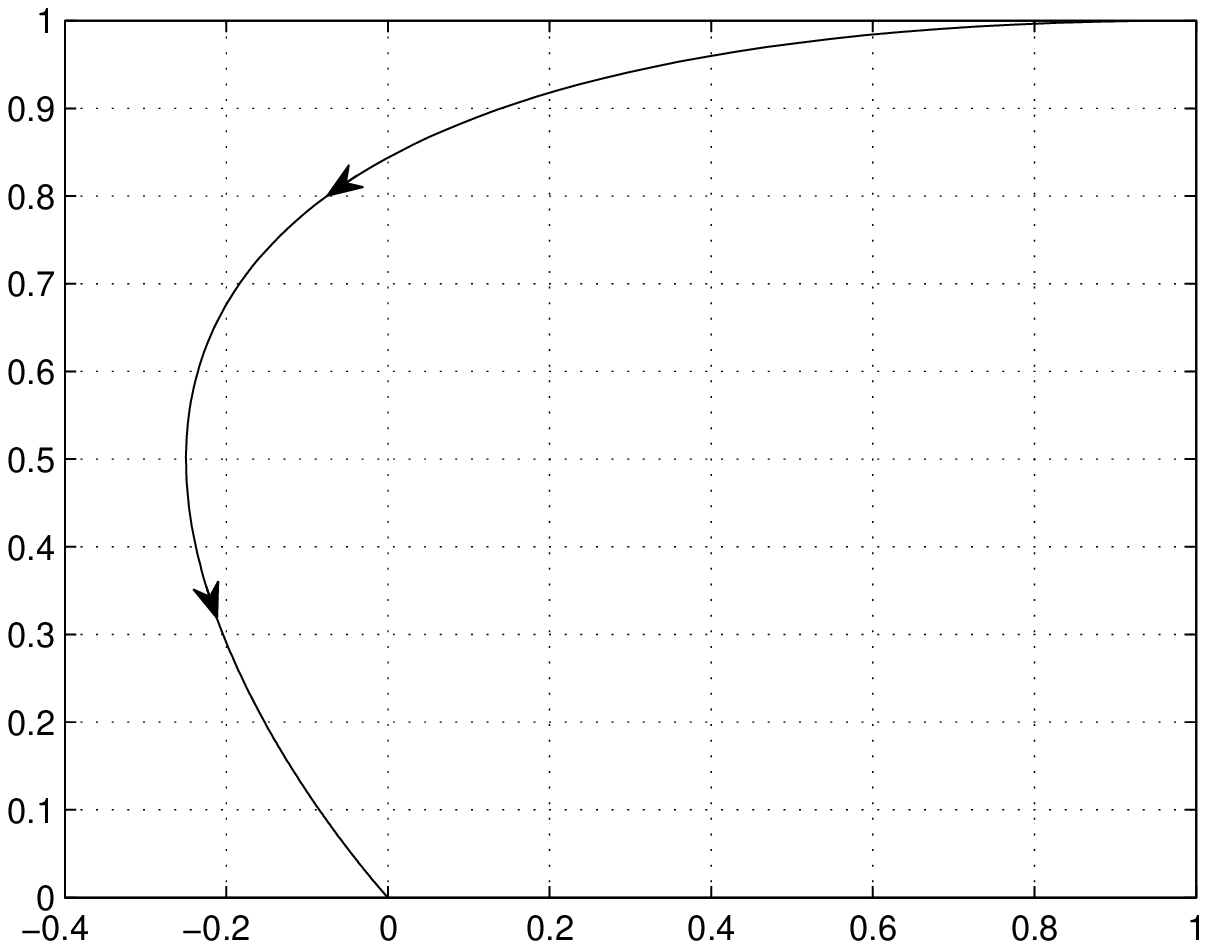}
\end{minipage}
}
\subfloat[Curvature]{
\label{fig:2(2)}
\begin{minipage}[h]{0.48\textwidth}
   \centering
   \includegraphics[angle=0,width=1\textwidth]{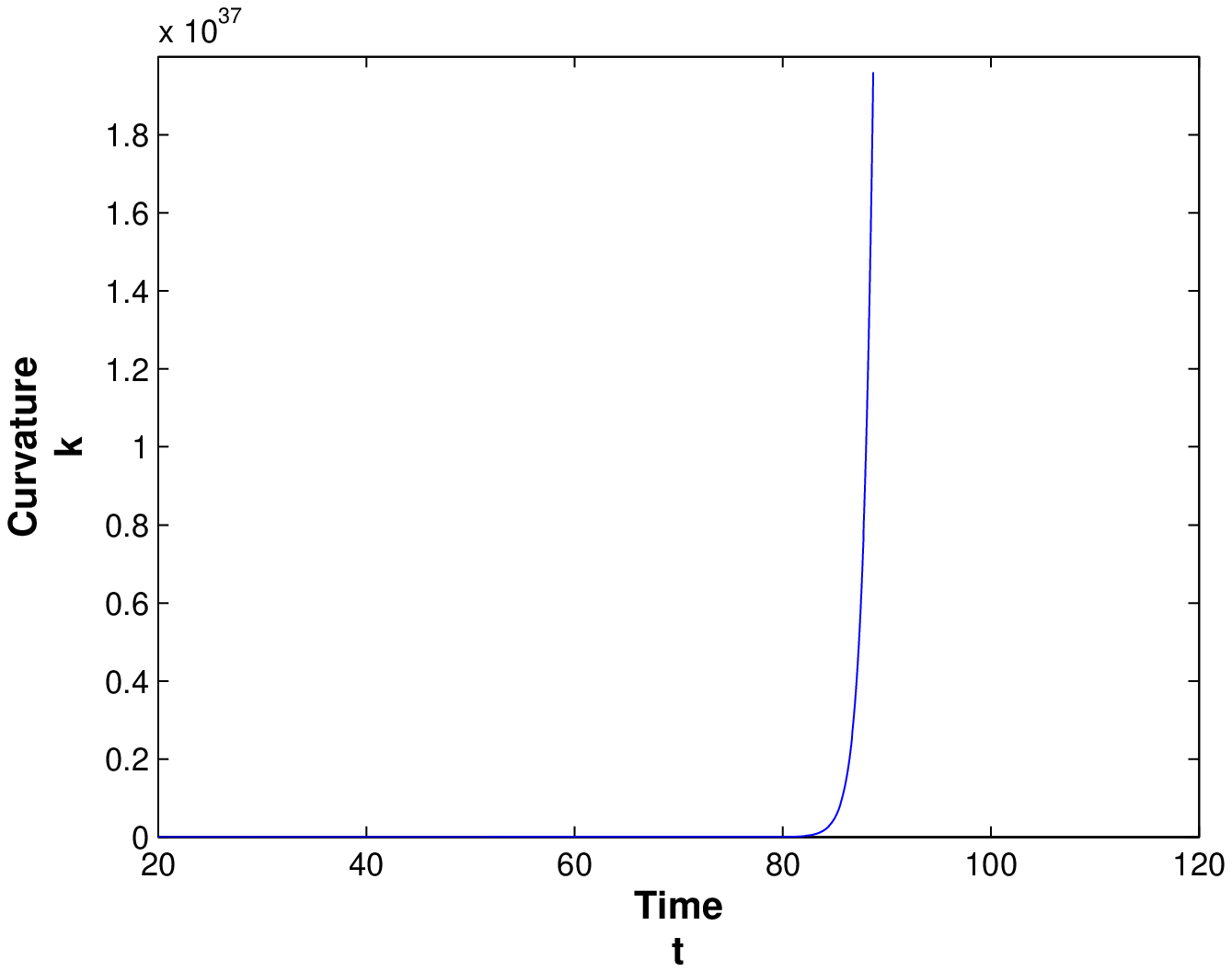}
\end{minipage}
}
\caption{Example 2}
\label{fig:2}
\end{figure}

\subsection*{Example 3 (Theorem \ref{thm1}(1))}

\

Let
\begin{align*}
\begin{pmatrix}\dot x\\[0.5ex] \dot y\\[0.5ex] \dot z\end{pmatrix}\
=\begin{pmatrix}-4 & -6 & 2\sqrt3 \\[0.5ex] -6 & -17 & -5\sqrt3 \\[0.5ex] 2\sqrt3 & -5\sqrt3 & -27 \end{pmatrix}\ \begin{pmatrix}x\\[0.5ex] y\\[0.5ex] z\end{pmatrix}\
\end{align*}
be a three-dimensional linear time-invariant system. Then we have
\begin{align*}
A&=\begin{pmatrix}-4 & -6 & 2\sqrt3 \\[0.5ex] -6 & -17 & -5\sqrt3 \\[0.5ex] 2\sqrt3 & -5\sqrt3 & -27 \end{pmatrix}\ , \\
\mathrm{e}^{tA}
&=\begin{pmatrix} \frac14(\mathrm{e}^{-16t}+3) & -\frac{3}{8}(-\mathrm{e}^{-16t}+1) & \frac{\sqrt3}{8}(-\mathrm{e}^{-16t}+1)
\\[2ex] -\frac{3}{8}(-\mathrm{e}^{-16t}+1) & \frac{1}{16}(4\mathrm{e}^{-32t}+9\mathrm{e}^{-16t}+3) & -\frac{\sqrt3}{16}(-4\mathrm{e}^{-32t}+3\mathrm{e}^{-16t}+1)
\\[2ex] \frac{\sqrt3}{8}(-\mathrm{e}^{-16t}+1) & -\frac{\sqrt3}{16}(-4\mathrm{e}^{-32t}+3\mathrm{e}^{-16t}+1) & \frac{1}{16}(12\mathrm{e}^{-32t}+3\mathrm{e}^{-16t}+1)
\end{pmatrix}.
\end{align*}
The square of curvature $\kappa(t)$ of curve $r(t)$ is
\begin{align*}
\kappa^2(t)&=\left( \frac{\left\|\dot{r}(t)\times\ddot{r}(t)\right\|}{\left\|\dot{r}(t)\right\|^3} \right)^2
=\left( \frac{\left\|A\mathrm{e}^{tA}r(0)\times A^2\mathrm{e}^{tA}r(0)\right\|}{\left\|A\mathrm{e}^{tA}r(0)\right\|^3} \right)^2 \nonumber\\
&=\frac{256 u^2 v^2 \mathrm{e}^{96t}}{ \{ \frac14 u^2 \mathrm{e}^{32t} + 3( \frac14 u \mathrm{e}^{16t}-2v )^2 + ( \frac34 u \mathrm{e}^{16t}+2v )^2 \}^3 },
\end{align*}
where $u=2x_0+3y_0-\sqrt3 z_0$, and $v=y_0+\sqrt3 z_0$.

If the initial value $r(0)=(x_0,y_0,z_0)^\mathrm{T}$ satisfies $u\neq0$ and $v\neq0$, then we have
\begin{align*}
\lim\limits_{t\to+\infty}\kappa(t)
=\frac{16 |v|}{u^2}
=\frac{16 |y_0+\sqrt3 z_0|}{(2x_0+3y_0-\sqrt3 z_0)^2}.
\end{align*}
Using Theorem \ref{thm1} (1), we obtain that the zero solution of the system is stable.

The trajectory of $r(t)$ and the graph of function $\kappa(t)$ are shown in Figure \ref{fig:3}, where $r(0)=(1,1,1)^\mathrm{T}$.
\begin{figure}[h]
\centering
\subfloat[Trajectory]{
\label{fig:3(1)}
\begin{minipage}[h]{0.48\textwidth}
   \centering
   \includegraphics[angle=0,width=1\textwidth]{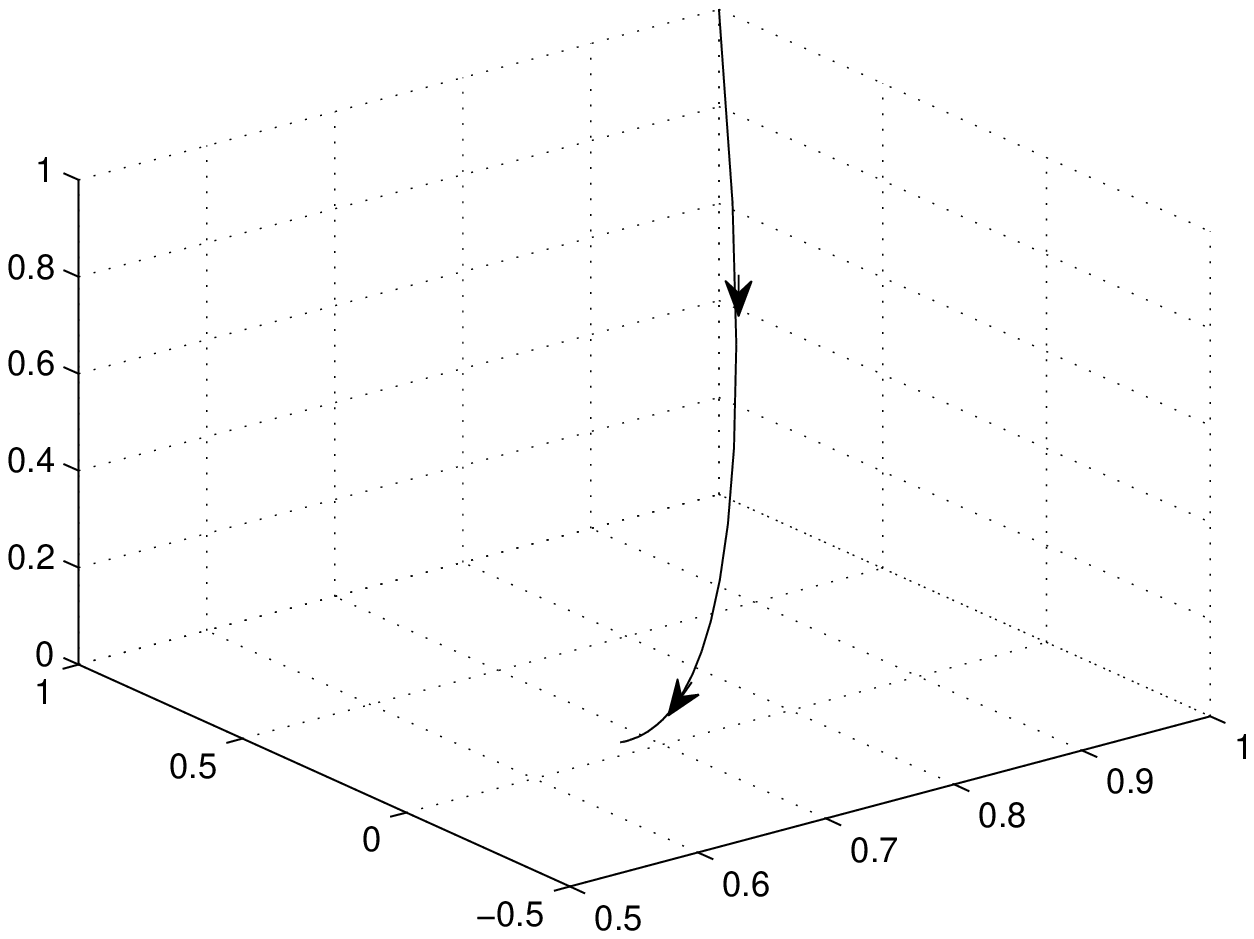}
\end{minipage}
}
\subfloat[Curvature]{
\label{fig:3(2)}
\begin{minipage}[h]{0.48\textwidth}
   \centering
   \includegraphics[angle=0,width=1\textwidth]{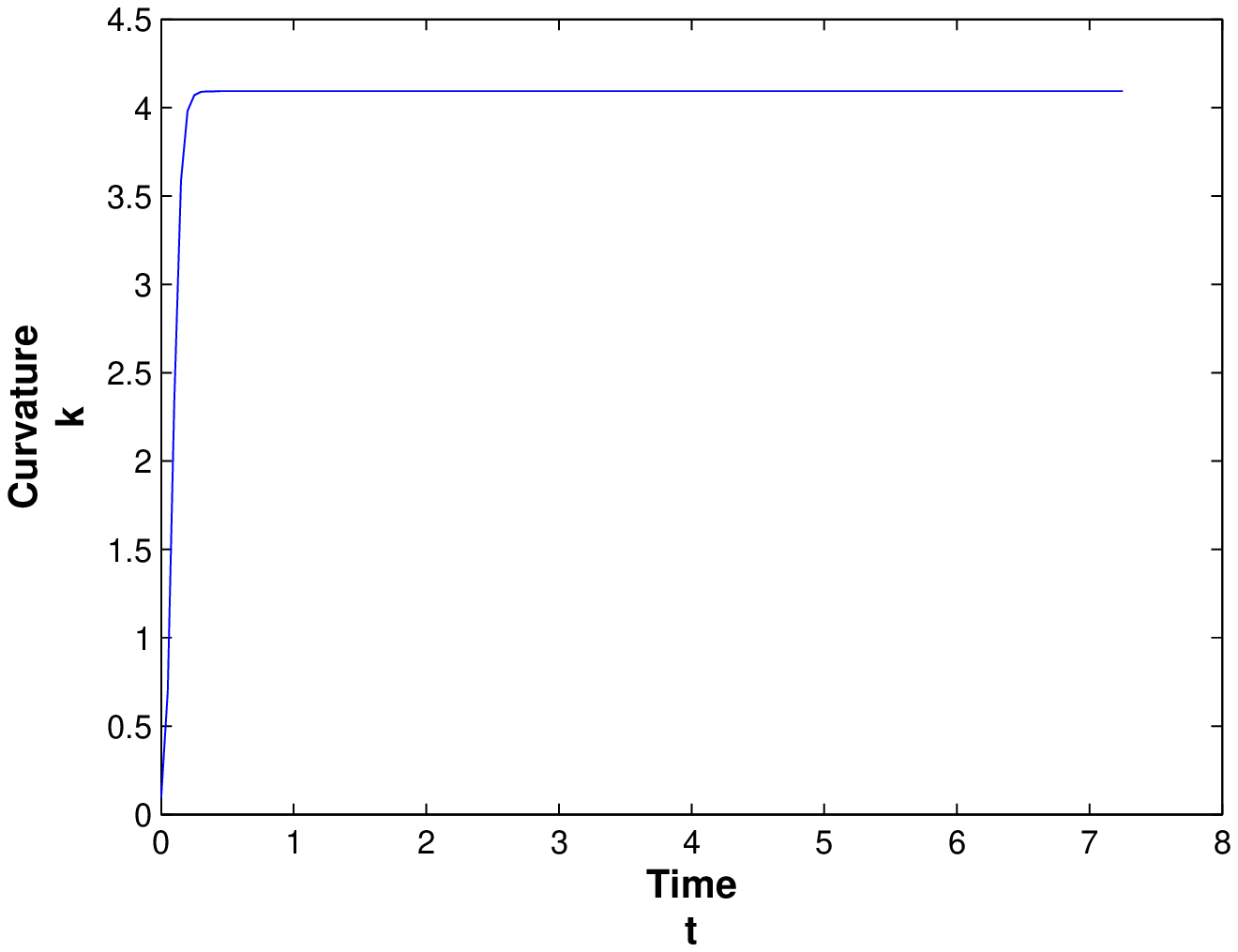}
\end{minipage}
}
\caption{Example 3}
\label{fig:3}
\end{figure}

\subsection*{Example 4 (Theorem \ref{thm1}(2))}

\

Let
\begin{align*}
\begin{pmatrix}\dot x\\[0.5ex] \dot y\\[0.5ex] \dot z\end{pmatrix}\
=\begin{pmatrix}-1&4&0\\[0.5ex] 0&-1&2\\[0.5ex] 0&-2&-1\end{pmatrix}\ \begin{pmatrix}x\\[0.5ex]y\\[0.5ex]z\end{pmatrix}\
\end{align*}
be a three-dimensional linear time-invariant system. Then we have
\begin{align*}
A=\begin{pmatrix}-1&4&0\\[0.5ex]0&-1&2\\[0.5ex]0&-2&-1\end{pmatrix}\ , \qquad
\mathrm{e}^{tA}
=\mathrm{e}^{-t}
\begin{pmatrix}1&2\sin{2t}&2(1-\cos{2t})
\\[1ex] 0&\cos{2t}&\sin{2t}
\\[1ex] 0&-\sin{2t}&\cos{2t}
\end{pmatrix}.
\end{align*}
The square of curvature $\kappa(t)$ of curve $r(t)$ is
\begin{align*}
\kappa^2(t)
=\frac{
20\mathrm{e}^{2t}(y_0^2+z_0^2)  \big\{ 5(y_0^2+z_0^2)[5-4\sin^2(2t-\varphi)] + [ x_0+2z_0+2\sqrt{5(y_0^2+z_0^2)}\sin(2t-\varphi) ]^2  \big\}
}{
\big\{
5(y_0^2+z_0^2)+ [ x_0+2z_0+2\sqrt{5(y_0^2+z_0^2)}\sin(2t-\varphi) ]^2
\big\}^3},
\end{align*}
where $r(0)=(x_0,y_0,z_0)^\mathrm{T}\in\mathbb{R}^3$ is the initial value of $r(t)$, and $\varphi\in\mathbb{R}$ satisfies $\sin{\varphi}=\frac{2y_0+z_0}{\sqrt{5(y_0^2+z_0^2)}}$ and $\cos{\varphi}=\frac{y_0-2z_0}{\sqrt{5(y_0^2+z_0^2)}}$.

If the initial value $r(0)=(x_0,y_0,z_0)^\mathrm{T}$ satisfies $y_0^2+z_0^2>0$, then we have $\lim\limits_{t\to+\infty}\kappa(t)=+\infty$. Noting that $\mathrm{det}\,A=-5\neq0$, and using Theorem \ref{thm1} (2), we obtain that the zero solution of the system is asymptotically stable.

The trajectory of $r(t)$ and the graph of function $\kappa(t)$ are shown in Figure \ref{fig:4}, where $r(0)=(1,1,1)^\mathrm{T}$.
\begin{figure}[h]
\centering
\subfloat[Trajectory]{
\label{fig:4(1)}
\begin{minipage}[h]{0.48\textwidth}
   \centering
   \includegraphics[angle=0,width=1\textwidth]{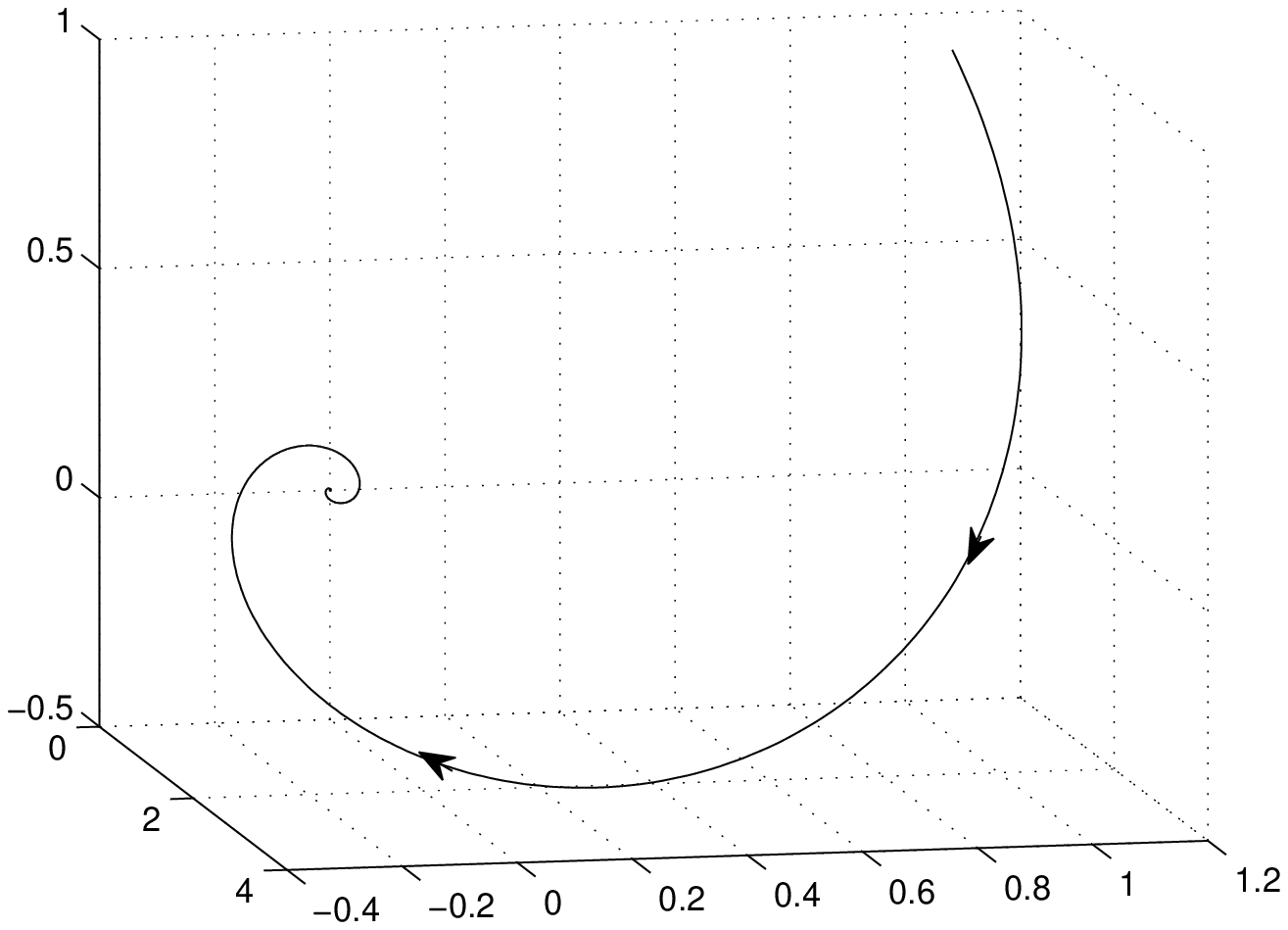}
\end{minipage}
}
\subfloat[Curvature]{
\label{fig:4(2)}
\begin{minipage}[h]{0.48\textwidth}
   \centering
   \includegraphics[angle=0,width=1\textwidth]{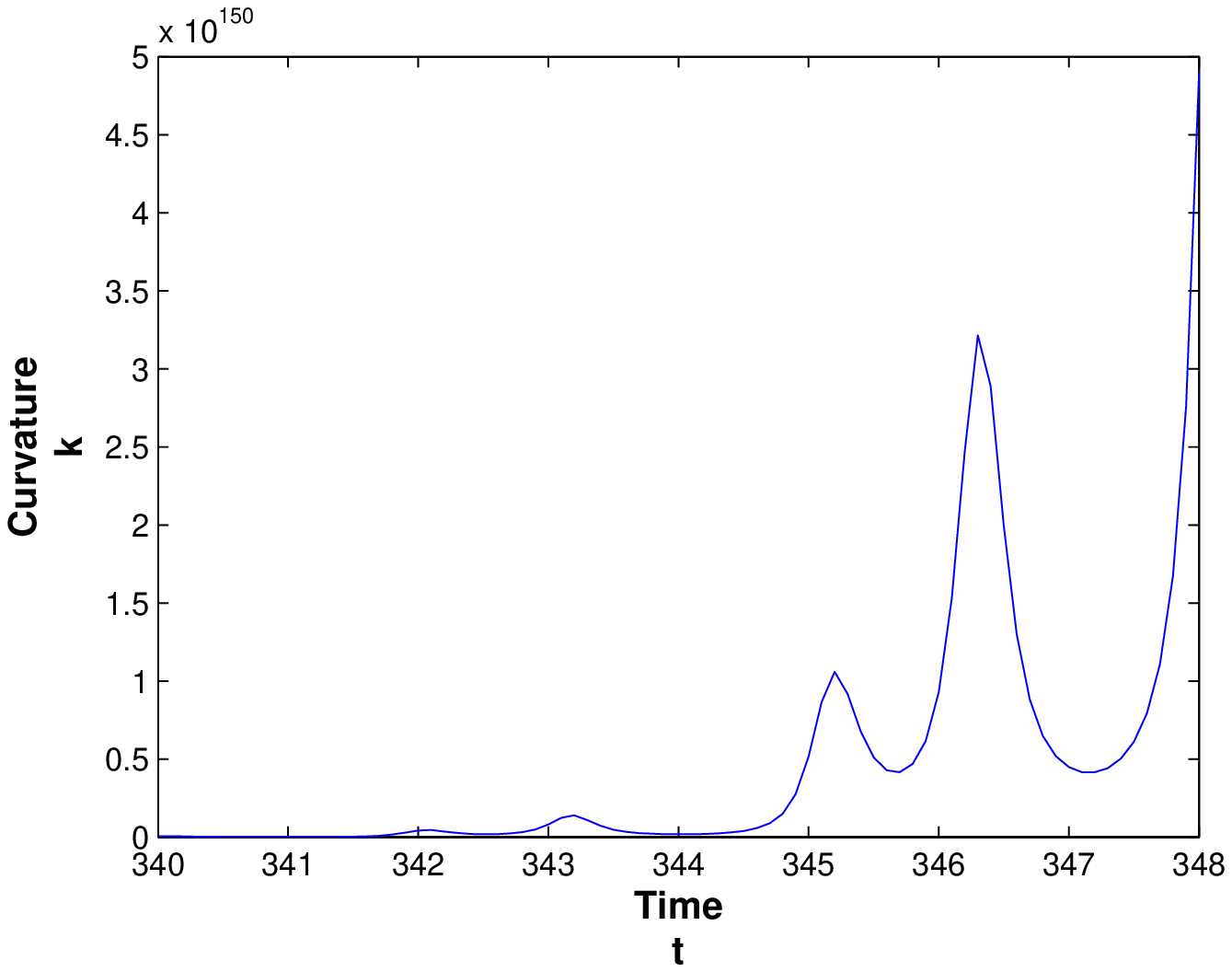}
\end{minipage}
}
\caption{Example 4}
\label{fig:4}
\end{figure}

\subsection*{Example 5 (Theorem \ref{thm1}(3))}

\

Let
\begin{align*}
\begin{pmatrix}\dot x\\[0.5ex]\dot y\\[0.5ex]\dot z\end{pmatrix}\
=\begin{pmatrix}-5&-12&9\\[0.5ex]-1&-30&19\\[0.5ex]0&-36&22\end{pmatrix}\ \begin{pmatrix}x\\[0.5ex]y\\[0.5ex]z\end{pmatrix}\
\end{align*}
be a three-dimensional linear time-invariant system. Then we have
\begin{align*}
A=\begin{pmatrix}-5&-12&9\\[0.5ex]-1&-30&19\\[0.5ex]0&-36&22\end{pmatrix}\ , \qquad
\mathrm{e}^{tA}
=\mathrm{e}^{-6t}
\begin{pmatrix}3 - 3 \mathrm{e}^t + \mathrm{e}^{4t} &~ 3(1 - \mathrm{e}^{4t}) &~ -3 + \mathrm{e}^t + 2 \mathrm{e}^{4t}
\\[1ex] 7 - 9 \mathrm{e}^t + 2 \mathrm{e}^{4t} &~ 7 - 6 \mathrm{e}^{4t} &~  -7 + 3 \mathrm{e}^t + 4 \mathrm{e}^{4t}
\\[1ex] 3(3 -4 \mathrm{e}^t + \mathrm{e}^{4t}) &~ 9(1 - \mathrm{e}^{4t}) &~ -9 + 4 \mathrm{e}^t + 6 \mathrm{e}^{4t}
\end{pmatrix}.
\end{align*}

We notice that $\lim\limits_{t\to+\infty}\kappa(t)=0$~ for almost all initial value $r(0)\in\mathbb{R}^3$ and therefore we cannot determine the stability of the system by curvature. Nevertheless, torsion $\tau(t)$ of curve $r(t)$ is
\begin{align*}
\tau(t)
=-\frac{20uvw \mathrm{e}^{9t}}
{25v^2(-3w+u\mathrm{e}^{4t})^2 + ( -10vw+8uw\mathrm{e}^{3t}+5uv\mathrm{e}^{4t} )^2 + [ 24uw\mathrm{e}^{3t}+5v(w+u\mathrm{e}^{4t}) ]^2},
\end{align*}
where $u=x_0-3y_0+2z_0$, $v=-3x_0+z_0$, and $w=x_0+y_0-z_0$.

If the initial value $r(0)=(x_0,y_0,z_0)^\mathrm{T}$ satisfies $uvw\neq0$, namely,
$
\left\{
\begin{aligned}
x_0-3y_0+2z_0 &\neq0,\\
-3x_0+z_0 &\neq0,\\
x_0+y_0-z_0 &\neq0,
\end{aligned}\right.
$
then we have
$
\lim\limits_{t\to+\infty}\tau(t)=\infty.
$
Using Theorem \ref{thm1} (3), we obtain that the zero solution of the system is asymptotically stable.

The trajectory of $r(t)$ and the graph of function $\tau(t)$ are shown in Figure \ref{fig:5}, where $r(0)=(2,1,1)^\mathrm{T}$.
\begin{figure}[h]
\centering
\subfloat[Trajectory]{
\label{fig:5(1)}
\begin{minipage}[h]{0.48\textwidth}
   \centering
   \includegraphics[angle=0,width=1\textwidth]{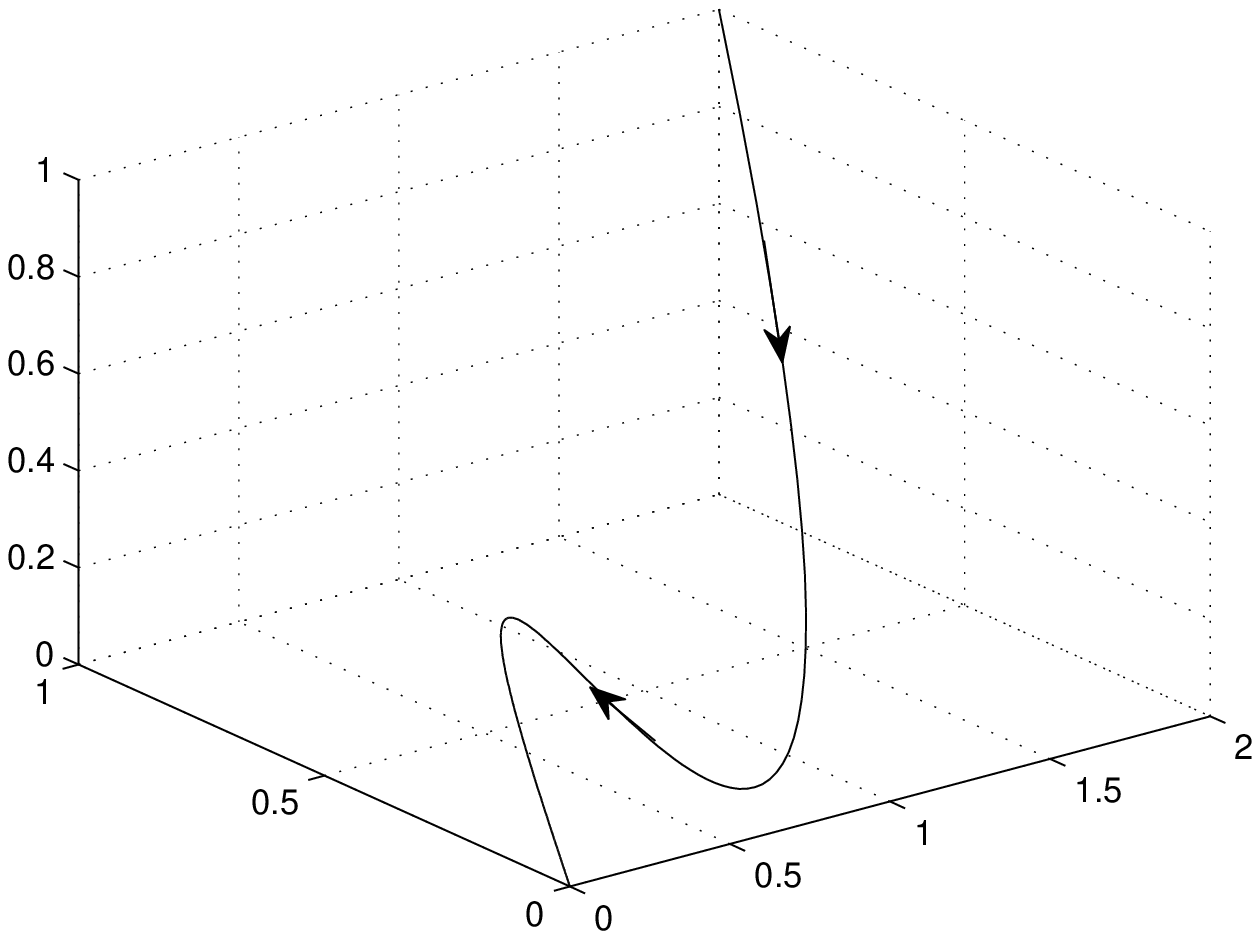}
\end{minipage}
}
\subfloat[Torsion]{
\label{fig:5(2)}
\begin{minipage}[h]{0.48\textwidth}
   \centering
   \includegraphics[angle=0,width=1\textwidth]{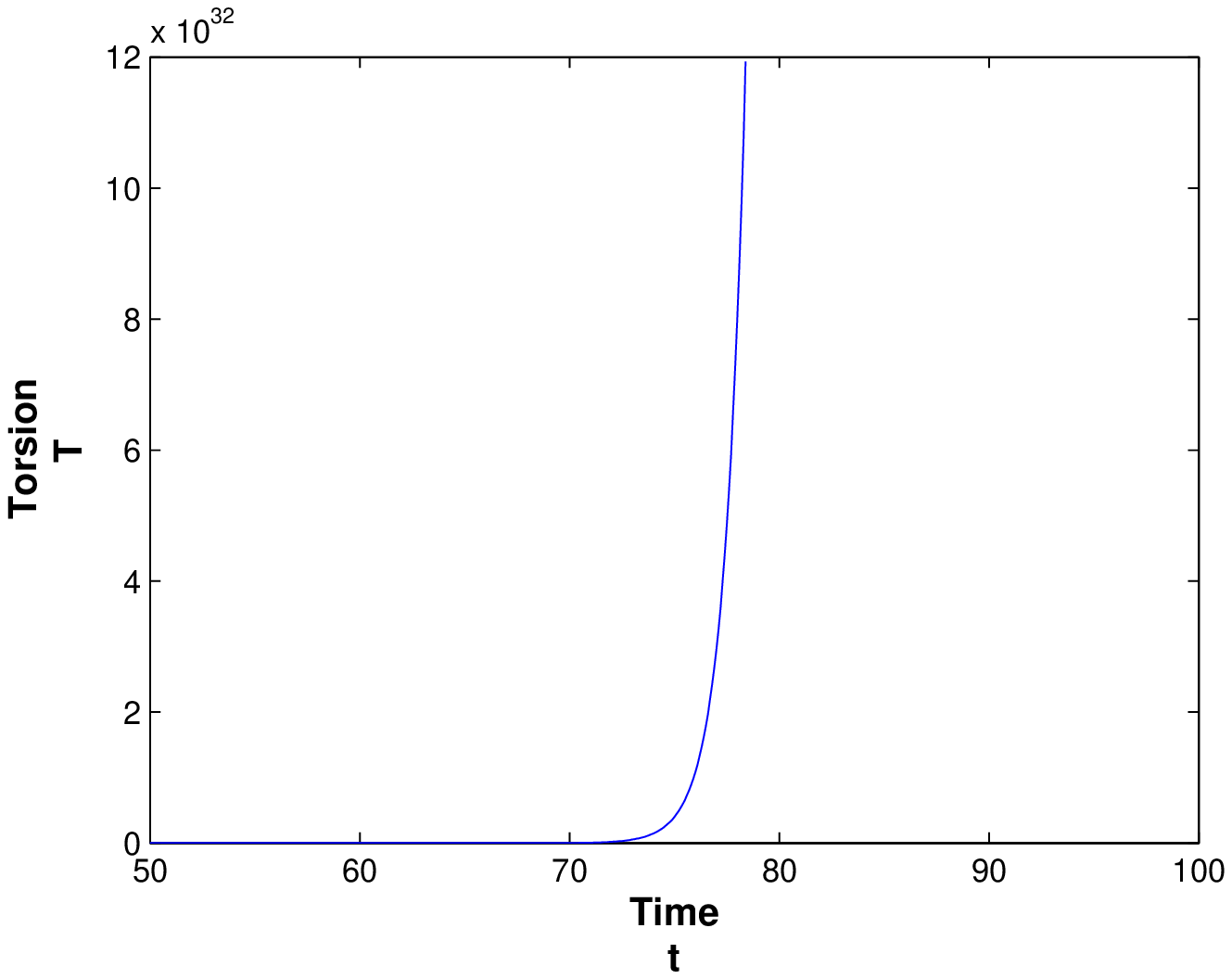}
\end{minipage}
}
\caption{Example 5}
\label{fig:5}
\end{figure}

\section{Conclusion}
The main results of this paper, Theorem \ref{thm 2-dim} and \ref{thm1}, are proved. Firstly, we give the relationship between curvatures of trajectories of two equivalent systems. Secondly, for each case of real Jordan canonical forms, we analyse curvatures and torsions of trajectories for all situations of the eigenvalues, which completes the proofs of theorems.
\par
These two theorems give the relationship between curvature and stability. More precisely, several sufficient conditions for stability of the zero solution of two or three-dimensional linear time-invariant systems, based on curvature and torsion, are given. For each case of the two theorems, we give an example to illustrate the results.
\par
Further work: we will give the description of stability for higher dimensional linear time-invariant systems, by curvatures, which may have some advantages for analysing the stability. Moreover, the application will be investigated in the future.

\section*{Acknowledgment}
The second author would like to express his sincere thanks to Professor D. Krupka for his helps. The special thanks to Professor Shoudong Huang of University of Technology, Sydney for his very helpful suggestions.

\
\appendix

\section{Proofs of Theorem \ref{thm k1} and \ref{thm k2}}\label{Proof}
We give the proofs of Theorem \ref{thm k1} and \ref{thm k2} in this section.
\par
First, we need the following concepts and lemmas.

\begin{defn}[\!\!\cite{Horn}]
Let $A$ be an $m\times n$ complex matrix with rank $r$, and $\lambda_1, \lambda_2, \cdots, \lambda_r$ be the non-zero eigenvalues of $AA^\mathrm{H}$, where $A^\mathrm{H}$ denotes the conjugate transpose of $A$. Then
\begin{align*}
\delta_i=\sqrt{\lambda_i}\quad (i=1,2,\cdots,r)
\end{align*}
are called the singular values of $A$.
\end{defn}

\begin{prop}[Singular value decomposition, the case of real matrix \cite{Horn}]\label{SVD}
Let $A$ be an $m\times n$ real matrix with rank $r$, and $\delta_1\geqslant\delta_2\geqslant\cdots\geqslant\delta_r$ be the singular values of $A$. Then there exists an $m\times m$ orthogonal matrix $U$ and an $n\times n$ orthogonal matrix $V$, such that
\begin{align*}
A=UDV^\mathrm{T}=U\begin{pmatrix}\Delta&0\\0&0\end{pmatrix}V^\mathrm{T},
\end{align*}
where $\Delta=\mathrm{diag}\{\delta_1, \delta_2, \cdots, \delta_r\}$.
\end{prop}

\begin{lem}[\!\!\cite{Strang}]\label{prop det}
Let $P$ be a $3\times 3$ matrix, and
$a,b,c\in\mathbb{R}^3$ be three column vectors.
Then the scalar triple product
\begin{align*}
(Pa, Pb, Pc)=(\mathrm{det}\, P)(a,b,c).
\end{align*}
\end{lem}

\begin{lem}\label{lem orth}
Let $U$ be a $3\times 3$ orthogonal matrix, and
$a, b\in \mathbb{R}^3$ be two column vectors. Then we have
\begin{align*}
\|Ua\times Ub\|=\|a\times b\|.
\end{align*}
\end{lem}

\begin{proof}
The proof is trivial and we omit here.
\end{proof}

\begin{lem}\label{lem k}
Let $\Delta$ be a $3\times 3$ diagonal matrix $\Delta=\mathrm{diag}\{\delta_1,\delta_2,\delta_3\}\ (\delta_1\geqslant\delta_2\geqslant\delta_3>0)$,
and $a, b\in \mathbb{R}^3$ be two column vectors. Then we have
\begin{align} \label{ieq1}
\delta_3\|a\| \leqslant \|\Delta a\| \leqslant \delta_1\|a\|,
\end{align}
\begin{align}\label{ieq2}
\delta_3^2\|a\times b\| \leqslant \|\Delta a\times \Delta b\| \leqslant \delta_1^2\|a\times b\|.
\end{align}
\end{lem}

\begin{proof}
Suppose that $a=(a_1, a_2, a_3)^\mathrm{T}$, and $b=(b_1, b_2, b_3)^\mathrm{T}$.
Noting that $\delta_1\geqslant\delta_2\geqslant\delta_3$, we have
\begin{align*}
\|\Delta a\|
=\left\|\begin{pmatrix}\delta_1 a_1, \delta_2 a_2, \delta_3 a_3\end{pmatrix} ^\mathrm{T} \right\|
=\sqrt{\sum_{i=1}^3 (\delta_i a_i)^2}
\leqslant \delta_1 \sqrt{\sum_{i=1}^3 a_i^2}
=\delta_1 \|a\|.
\end{align*}
Similarly, we have
$
\|\Delta a\| \geqslant \delta_3 \|a\|.
$
Thereby, inequality (\ref{ieq1}) is proved.
\par
To prove inequality (\ref{ieq2}), we have
\begin{align*}
\|\Delta a \times \Delta b \|
&=\left\|\begin{pmatrix}\delta_1 a_1\\\delta_2 a_2\\\delta_3 a_3\end{pmatrix}
\times \begin{pmatrix}\delta_1 b_1\\\delta_2 b_2\\\delta_3 b_3\end{pmatrix}\right\|
=\left\|\begin{pmatrix}\delta_2\delta_3 \begin{vmatrix}a_2 & b_2\\a_3 & b_3\end{vmatrix},&
\delta_1\delta_3 \begin{vmatrix}a_3 & b_3\\a_1 & b_1\end{vmatrix},&
\delta_1\delta_2 \begin{vmatrix}a_1 & b_1\\a_2 & b_2\end{vmatrix}
\end{pmatrix}^\mathrm{T}\right\|
\nonumber\\
&\leqslant\delta_1^2\left\|\begin{pmatrix}\begin{vmatrix}a_2 & b_2\\a_3 & b_3\end{vmatrix},&
\begin{vmatrix}a_3 & b_3\\a_1 & b_1\end{vmatrix},&
\begin{vmatrix}a_1 & b_1\\a_2 & b_2\end{vmatrix}
\end{pmatrix}^\mathrm{T}\right\|
=\delta_1^2\left\|a\times b\right\|.
\end{align*}
Similarly, we have
$
\|\Delta a \times \Delta b \|\geqslant\delta_3^2\left\|a\times b\right\|.
$
\end{proof}

Now, we proceed to prove Theorem \ref{thm k1}.
\begin{proof}[Proof of Theorem \ref{thm k1}]
Recall that $A=P^{-1}BP$, where $P$ is a $3\times 3$ real invertible matrix. Using Proposition \ref{SVD}, we obtain a singular value decomposition of $P$, namely,
\begin{align*}
P=U\Delta V^\mathrm{T},
\end{align*}
where $U$ and $V$ are $3\times 3$ orthogonal matrices, and $\Delta=\mathrm{diag}\{\delta_1, \delta_2, \delta_3\}\ ( \delta_1\geqslant\delta_2\geqslant\delta_3>0)$. Since $\kappa_v(t)$ is curvature of trajectory $v(t)$ of the system $\dot{v}(t)=Bv(t)$, we have
\begin{align*}
\kappa_v(t)=\frac{\|\dot v\times\ddot v\|}{\|\dot v\|^3}
=\frac{\|(P\dot r)\times(P\ddot r)\|}{\|P\dot r\|^3}
=\frac{\|(U\Delta V^\mathrm{T}\dot r)\times(U\Delta V^\mathrm{T}\ddot r)\|}{\|U\Delta V^\mathrm{T}\dot r\|^3}.
\end{align*}
Noting that $U$ is an orthogonal matrix, we have $\|U\Delta V^\mathrm{T}\dot r\|=\|\Delta V^\mathrm{T}\dot r\|$, and
$\|(U\Delta V^\mathrm{T}\dot r)\times(U\Delta V^\mathrm{T}\ddot r)\|
=\|\Delta V^\mathrm{T}\dot r\times\Delta V^\mathrm{T}\ddot r\|$,
by using Lemma \ref{lem orth}. Thus
\begin{align*}
\kappa_v(t)
=\frac{\|\Delta V^\mathrm{T}\dot r\times\Delta V^\mathrm{T}\ddot r\|}{\|\Delta V^\mathrm{T}\dot r\|^3}.
\end{align*}
Next, using Lemma \ref{lem k}, we obtain two inequalities
\begin{align*}
\delta_3^2\|V^\mathrm{T}\dot r\times V^\mathrm{T}\ddot r\|
\leqslant \|\Delta V^\mathrm{T}\dot r\times\Delta V^\mathrm{T}\ddot r\|
\leqslant \delta_1^2\|V^\mathrm{T}\dot r\times V^\mathrm{T}\ddot r\|,
\end{align*}
\begin{align*}
\delta_3\|V^\mathrm{T}\dot r\|
\leqslant \|\Delta V^\mathrm{T}\dot r\|
\leqslant \delta_1\|V^\mathrm{T}\dot r\|.
\end{align*}
Therefore we have
\begin{align*}
\frac{\delta_3^2\|V^\mathrm{T}\dot r\times V^\mathrm{T}\ddot r\| }{\delta_1^3\|V^\mathrm{T}\dot r\|^3}
\leqslant \kappa_v(t)
\leqslant \frac{\delta_1^2\|V^\mathrm{T}\dot r\times V^\mathrm{T}\ddot r\|}{\delta_3^3\|V^\mathrm{T}\dot r\|^3}.
\end{align*}
Because $V$ is an orthogonal matrix, we have $\|V^\mathrm{T}\dot r\|=\|\dot r\|$, and
$\|V^\mathrm{T}\dot r\times V^\mathrm{T}\ddot r\|=\|\dot r\times \ddot r\|$.
Hence we obtain
\begin{align*}
\frac{\delta_3^2}{\delta_1^3}\,\kappa_r(t)
\leqslant \kappa_v(t)
\leqslant \frac{\delta_1^2}{\delta_3^3}\,\kappa_r(t).
\end{align*}
It follows that
\begin{align*} 
\lim\limits_{t\to+\infty}\kappa_v(t)=0 \iff &\lim\limits_{t\to+\infty}\kappa_r(t)=0, \nonumber\\
\lim\limits_{t\to+\infty}\kappa_v(t)=+\infty \iff &\lim\limits_{t\to+\infty}\kappa_r(t)=+\infty, \\
\exists C_1, C_2>0, \ \exists T>0, \ \mathrm{s.t.},&\ \kappa_v(t)\in\left[C_1, C_2\right], \forall t>T \nonumber\\
\iff \exists \tilde{C_1}, \tilde{C_2}>0, \ \exists \tilde{T}>0, \ \mathrm{s.t.},&\ \kappa_r(t)\in\left[\tilde{C_1}, \tilde{C_2}\right], \forall t>\tilde{T}.
\nonumber
\end{align*}
This completes the proof of Theorem \ref{thm k1}.
\end{proof}

Next we give the proof of Theorem \ref{thm k2}.
\begin{proof}[Proof of Theorem \ref{thm k2}]
Recall that $A=P^{-1}BP$, and we suppose that $P$ has the singular value decomposition $P=U\Delta V^\mathrm{T}=U\mathrm{diag}\{\delta_1, \delta_2, \delta_3\}V^\mathrm{T}\ ( \delta_1\geqslant\delta_2\geqslant\delta_3>0)$.
Since $\tau_v(t)$ is torsion of trajectory $v(t)$ of the system $\dot{v}(t)=Bv(t)$, using Lemma \ref{prop det}, we have
\begin{align*}
\tau_v(t)=\frac{\left(\dot{v}(t),\ddot{v}(t),\dddot{v}(t)\right)}{\left\|\dot{v}(t)\times\ddot{v}(t)\right\|^2}
=\frac{\left(P\dot{r}(t),P\ddot{r}(t),P\dddot{r}(t)\right)}{\left\|\dot{v}(t)\times\ddot{v}(t)\right\|^2}
=\frac{(\mathrm{det}\,P)\left(\dot{r}(t),\ddot{r}(t),\dddot{r}(t)\right)}{\left\|\dot{v}(t)\times\ddot{v}(t)\right\|^2}.
\end{align*}
According to the proof of Theorem \ref{thm k1}, we have
\begin{align*}
\delta_3^2\left\|\dot{r}(t)\times\ddot{r}(t)\right\|
\leqslant \left\|\dot{v}(t)\times\ddot{v}(t)\right\|
\leqslant \delta_1^2\left\|\dot{r}(t)\times\ddot{r}(t)\right\|.
\end{align*}
Thus, we see that
if $\mathrm{det}\,P>0$, then $\tau_v(t)$ and $\tau_r(t)$ are of same sign;
if $\mathrm{det}\,P<0$, then $\tau_v(t)$ and $\tau_r(t)$ are of opposite sign.
Hence we have
\begin{align*}
\frac{|\mathrm{det}\,P|}{\delta_1^4}|\tau_r(t)|
\leqslant |\tau_v(t)|
\leqslant \frac{|\mathrm{det}\,P|}{\delta_3^4}|\tau_r(t)|.
\end{align*}
It follows that
\begin{align*}
\lim\limits_{t\to+\infty}\tau_v(t)=0 \iff &\lim\limits_{t\to+\infty}\tau_r(t)=0, \nonumber\\
\lim\limits_{t\to+\infty}\tau_v(t)=\infty \iff &\lim\limits_{t\to+\infty}\tau_r(t)=\infty, \\
\exists C_1, C_2>0, \ \exists T>0, \ \mathrm{s.t.},&\ \left|\tau_v(t)\right|\in\left[C_1, C_2\right], \forall t>T. \nonumber\\
\iff \exists \tilde{C_1}, \tilde{C_2}>0, \ \exists \tilde{T}>0, \ \mathrm{s.t.},&\ \left|\tau_r(t)\right|\in\left[\tilde{C_1}, \tilde{C_2}\right], \forall t>\tilde{T}.\nonumber
\end{align*}
This completes the proof of Theorem \ref{thm k2}.
\end{proof}

\section{Curvatures of Two-Dimensional Systems}\label{cal 2-dim}
We give the calculations of curvature of two-dimensional systems in three cases respectively.

\subsection{Case 1}

\

In the case of
\begin{align*}
A=\begin{pmatrix}\lambda_1&0\\0&\lambda_2\end{pmatrix}\
(\lambda_1, \lambda_2\in\mathbb{R}),
\end{align*}
by Proposition \ref{ODE}, we know that $r(t)=\mathrm{e}^{tA}r(0)$ is the solution of $\dot{r}(t)=Ar(t)$, where $r(0)=(x_0,y_0)^\mathrm{T}$ is the initial value of $r(t)$.
The first and second derivative of $r(t)$ are
\begin{align*}
\dot{r}(t)=Ar(t)
=\begin{pmatrix}\lambda_1&0\\0&\lambda_2\end{pmatrix}\begin{pmatrix}x(t)\\y(t)\end{pmatrix}
=\begin{pmatrix}\lambda_1 x\\\lambda_2 y\end{pmatrix},
\end{align*}
and
\begin{align*}
\ddot{r}(t)=A\dot{r}(t)
=\begin{pmatrix}\lambda_1&0\\0&\lambda_2\end{pmatrix}\begin{pmatrix}\lambda_1 x\\\lambda_2 y\end{pmatrix}
=\begin{pmatrix}\lambda_1^2 x\\\lambda_2^2 y\end{pmatrix},
\end{align*}
respectively.
\par
If $\lambda_1=\lambda_2=0$, namely, $A=0$, then every trajectory $r(t)=\mathrm{e}^{tA}r(0)=r(0)$ is a constant point, we have a convention that $\kappa(t)\equiv0$.
If $\lambda_1^2+\lambda_2^2\neq0$, then the square of curvature $\kappa(t)$ of curve $r(t)$ is
\begin{align}\label{k square 2-dim (1)}
\kappa^2(t)
=\kappa_s^2(t)
=\left( \frac{\dot x\ddot y-\ddot x\dot y}{ (\dot x^2+\dot y^2)^{3/2}} \right) ^2
=\frac{(\lambda_1x \cdot \lambda_2^2y - \lambda_1^2 x \cdot \lambda_2y)^2}
{\left\{(\lambda_1x)^2+(\lambda_2y)^2\right\}^3}
=\frac{\left\{\lambda_1 \lambda_2(\lambda_2-\lambda_1)xy\right\}^2}
{\left\{(\lambda_1x)^2+(\lambda_2y)^2\right\}^3},
\end{align}
where $\kappa_s(t)=\frac{\dot x(t)\ddot y(t)-\ddot x(t)\dot y(t)}{ (\dot x^2(t)+\dot y^2(t))^{3/2}}$ is the signed curvature of plane curve $r(t)$, and $\kappa_s(t)=\pm \kappa(t)$ (cf.\cite{Carmo}).

Noting that
\begin{align*}
r(t)=\mathrm{e}^{tA}r(0)
=\begin{pmatrix}\mathrm{e}^{\lambda_1t}&0
\\[0.5ex] 0&\mathrm{e}^{\lambda_2t}
\end{pmatrix}
\begin{pmatrix}x_0\\[0.5ex] y_0\end{pmatrix}
=\begin{pmatrix}x_0\mathrm{e}^{\lambda_1t}
\\[0.5ex] y_0\mathrm{e}^{\lambda_2t}
\end{pmatrix},
\end{align*}
we have
\begin{align}\label{r(t) 2-dim (1)}
\left\{
\begin{aligned}
x(t) &=x_0\mathrm{e}^{\lambda_1t},\\
y(t) &=y_0\mathrm{e}^{\lambda_2t}.
\end{aligned}\right.
\end{align}
By substituting (\ref{r(t) 2-dim (1)}) into (\ref{k square 2-dim (1)}), we obtain
\begin{align*}
\kappa^2(t)
=\frac{\left\{\lambda_1 \lambda_2(\lambda_2-\lambda_1)x_0y_0\right\}^2 \cdot \mathrm{e}^{2(\lambda_1+\lambda_2)t}}
{\left\{(\lambda_1x_0)^2 \cdot \mathrm{e}^{2\lambda_1t}+(\lambda_2y_0)^2 \cdot \mathrm{e}^{2\lambda_2t}\right\}^3}.
\end{align*}

\subsection{Case 2}

\

In the case of
\begin{align*}
A=\begin{pmatrix}a&b\\-b&a\end{pmatrix}\
(a, b\in\mathbb{R},\ b>0),
\end{align*}
the first and second derivative of $r(t)$ are
\begin{align*}
\dot{r}(t)&=Ar(t)=(ax+by,~ -bx+ay)^\mathrm{T},\\
\ddot{r}(t)&=A^2r(t)
=((a^2-b^2)x+2aby,~ -2abx+(a^2-b^2)y)^\mathrm{T},
\end{align*}
respectively.
Since
\begin{align*}
\mathrm{e}^{tA}
=\mathrm{e}^{at}
\begin{pmatrix}
\cos{bt} & \sin{bt}
\\[0.5ex] -\sin{bt} & \cos{bt}
\end{pmatrix},
\end{align*}
we have
\begin{align*}
\left\{
\begin{aligned}
x(t) &=\mathrm{e}^{at}(x_0 \cos{bt} + y_0 \sin{bt}),\\
y(t) &=\mathrm{e}^{at}(-x_0 \sin{bt} + y_0 \cos{bt}).
\end{aligned}\right.
\end{align*}
Thereby,
\begin{align*} 
x^2(t)+y^2(t) = (x_0^2+y_0^2)\mathrm{e}^{2at}.
\end{align*}
We obtain
the square of curvature $\kappa(t)$ of curve $r(t)$ as
\begin{align*}
\kappa^2(t)
=\left( \frac{\dot x\ddot y-\ddot x\dot y}{ (\dot x^2+\dot y^2)^{3/2}} \right) ^2
=\frac{b^2}{(a^2+b^2)(x^2+y^2)}
=\frac{b^2}{(a^2+b^2)(x_0^2+y_0^2)\cdot \mathrm{e}^{2at}}.
\end{align*}

\subsection{Case 3}

\

In the case of
\begin{align*}
A=\begin{pmatrix}\lambda&1\\0&\lambda\end{pmatrix}\
(\lambda\in\mathbb{R}).
\end{align*}
the first and second derivative of $r(t)$ are
\begin{align*}
\dot{r}(t)&=Ar(t)=(\lambda x+y,~ \lambda y)^\mathrm{T},\\
\ddot{r}(t)&=A^2r(t)
=(\lambda^2 x+2\lambda y,~ \lambda^2 y))^\mathrm{T},
\end{align*}
respectively.
Since
\begin{align*}
\mathrm{e}^{tA}
=\mathrm{e}^{\lambda t}
\begin{pmatrix}1&t
\\[0.5ex] 0&1
\end{pmatrix},
\end{align*}
we have
\begin{align*}
\left\{
\begin{aligned}
x(t) &=(x_0+y_0 t)\mathrm{e}^{\lambda t},\\
y(t) &=y_0\mathrm{e}^{\lambda t}.
\end{aligned}\right.
\end{align*}
Hence the square of curvature $\kappa(t)$ of curve $r(t)$ is
\begin{align*}
\kappa^2(t)
&=\left( \frac{\dot x\ddot y-\ddot x\dot y}{ (\dot x^2+\dot y^2)^{3/2}} \right) ^2
=\frac{ \lambda^4 y^4}
{\left\{(\lambda x+y)^2+(\lambda y)^2\right\}^3} \\
&=\frac{ \lambda^4 y_0^4 \cdot \mathrm{e}^{4\lambda t}}
{\left\{(\lambda x_0+\lambda y_0 \cdot t+y_0)^2+(\lambda y_0)^2\right\}^3 \cdot \mathrm{e}^{6\lambda t} }
=\frac{ \lambda^4 y_0^4 }
{g(t) \cdot \mathrm{e}^{2\lambda t} },
\end{align*}
where $g(t)$ is a polynomial in $t$. If $\lambda\neq0$, then $g(t)= \left( \lambda^6 y_0^6 \right) \cdot t^{6} + \sum_{i=0}^{5} a_i t^i$ is a polynomial of degree 6 in $t$; if $\lambda=0$, then $g(t)=y_0^6$ is a constant.

\section{Curvatures of Three-Dimensional Systems} \label{cal 3-dim}

We give the calculations of curvature and torsion of three-dimensional systems in four cases respectively.

\subsection{Case 1}

\

In the case of
\begin{align*}
A=\begin{pmatrix}\lambda_1&0&0\\0&\lambda_2&0\\0&0&\lambda_3\end{pmatrix}\
(\lambda_1, \lambda_2, \lambda_3\in\mathbb{R}),
\end{align*}
by Proposition \ref{ODE}, we know that $r(t)=\mathrm{e}^{tA}r(0)$ is the solution of $\dot{r}(t)=Ar(t)$, where $r(0)=(x_0,y_0,z_0)^\mathrm{T}$ is the initial value of $r(t)$.
The first, second and third derivative of $r(t)$ are
\begin{align*}
\dot{r}(t)=\begin{pmatrix}\lambda_1 x,~ \lambda_2 y,~ \lambda_3 z\end{pmatrix} ^\mathrm{T}, \quad
\ddot{r}(t)=\begin{pmatrix}\lambda_1^2 x,~ \lambda_2^2 y,~ \lambda_3^2 z\end{pmatrix} ^\mathrm{T}, \quad
\dddot{r}(t)=\begin{pmatrix}\lambda_1^3 x,~ \lambda_2^3 y,~ \lambda_3^3 z\end{pmatrix} ^\mathrm{T},
\end{align*}
respectively.
The norm of $\dot{r}(t)$ and the vector product $\dot{r}(t)\times\ddot{r}(t)$ are
\begin{align*}
\left\|\dot{r}(t)\right\|
=\sqrt{(\lambda_1 x)^2+(\lambda_2 y)^2+(\lambda_3 z)^2},
\end{align*}
and
\begin{align*}
\left\|\dot{r}(t)\times\ddot{r}(t)\right\|
&=\left\|\begin{pmatrix}
\begin{vmatrix}\lambda_2 y & \lambda_2^2 y\\[1.2ex] \lambda_3 z & \lambda_3^2 z\end{vmatrix},&
\begin{vmatrix}\lambda_3 z & \lambda_3^2 z\\[1.2ex] \lambda_1 x & \lambda_1^2 x\end{vmatrix},&
\begin{vmatrix}\lambda_1 x & \lambda_1^2 x\\[1.2ex] \lambda_2 y & \lambda_2^2 y\end{vmatrix}
\end{pmatrix}^\mathrm{T}\right\| \nonumber\\
&=\sqrt{\left\{\lambda_2 \lambda_3 (\lambda_3-\lambda_2) yz\right\}^2
+\left\{\lambda_1 \lambda_3 (\lambda_1-\lambda_3) xz\right\}^2
+\left\{\lambda_1 \lambda_2 (\lambda_2-\lambda_1) xy\right\}^2}.
\end{align*}
The scalar triple product $\left(\dot{r}(t),\ddot{r}(t),\dddot{r}(t)\right)$ is
\begin{align*}
\left(\dot{r}(t),\ddot{r}(t),\dddot{r}(t)\right)
=\begin{vmatrix}
\lambda_1x & \lambda_1^2x & \lambda_1^3x \\[0.5ex]
\lambda_2y & \lambda_2^2y & \lambda_2^3y \\[0.5ex]
\lambda_3z & \lambda_3^2z & \lambda_3^3z
\end{vmatrix}
=\lambda_1\lambda_2\lambda_3(\lambda_2-\lambda_1)(\lambda_3-\lambda_1)(\lambda_3-\lambda_2) xyz.
\end{align*}
\par
If $\lambda_1=\lambda_2=\lambda_3=0$, namely, $A=0$, then every trajectory $r(t)=\mathrm{e}^{tA}r(0)=r(0)$ is a constant point, we have conventions that $\kappa(t)\equiv0$ and $\tau(t)\equiv0$.
If $\lambda_1^2+\lambda_2^2+\lambda_3^2\neq 0$, then the square of curvature $\kappa(t)$ of curve $r(t)$ is
\begin{align}\label{k square}
\kappa^2(t)
=\left( \frac{\left\|\dot{r}(t)\times\ddot{r}(t)\right\|}{\left\|\dot{r}(t)\right\|^3} \right)^2
=\frac{\left\{\lambda_2 \lambda_3 (\lambda_3-\lambda_2) yz\right\}^2
+\left\{\lambda_1 \lambda_3 (\lambda_1-\lambda_3) xz\right\}^2
+\left\{\lambda_1 \lambda_2 (\lambda_2-\lambda_1) xy\right\}^2}
{\{(\lambda_1 x)^2+(\lambda_2 y)^2+(\lambda_3 z)^2\}^3}.
\end{align}
\par
If $[\lambda_2 \lambda_3 (\lambda_3-\lambda_2) ]^2
+[\lambda_1 \lambda_3 (\lambda_1-\lambda_3) ]^2
+[\lambda_1 \lambda_2 (\lambda_2-\lambda_1) ]^2=0$,
then $\left\|\dot{r}(t)\times\ddot{r}(t)\right\|\equiv0$ and we have $\tau(t)\equiv0$.
If $\left\|\dot{r}(t)\times\ddot{r}(t)\right\|\neq0$, then the torsion of curve $r(t)$ is
\begin{align}\label{tau}
\tau(t)
=\frac{\left(\dot{r}(t),\ddot{r}(t),\dddot{r}(t)\right)}{\left\|\dot{r}(t)\times\ddot{r}(t)\right\|^2}
=\frac{\lambda_1\lambda_2\lambda_3(\lambda_2-\lambda_1)(\lambda_3-\lambda_1)(\lambda_3-\lambda_2) xyz}
{\left\{\lambda_2 \lambda_3 (\lambda_3-\lambda_2) yz\right\}^2
+\left\{\lambda_1 \lambda_3 (\lambda_1-\lambda_3) xz\right\}^2
+\left\{\lambda_1 \lambda_2 (\lambda_2-\lambda_1) xy\right\}^2}.
\end{align}
Noting that
\begin{align*}
r(t)=\mathrm{e}^{tA}r(0)
=\begin{pmatrix}\mathrm{e}^{\lambda_1t}&0&0
\\[0.5ex] 0&\mathrm{e}^{\lambda_2t}&0
\\[0.5ex] 0&0&\mathrm{e}^{\lambda_3t}
\end{pmatrix}
\begin{pmatrix}x_0\\[0.5ex] y_0\\[0.5ex] z_0\end{pmatrix}
=\begin{pmatrix}x_0\mathrm{e}^{\lambda_1t}
\\[0.5ex] y_0\mathrm{e}^{\lambda_2t}
\\[0.5ex] z_0\mathrm{e}^{\lambda_3t}\end{pmatrix},
\end{align*}
we have
\begin{align}\label{r(t)}
\left\{
\begin{aligned}
x(t) &=x_0\mathrm{e}^{\lambda_1t},\\
y(t) &=y_0\mathrm{e}^{\lambda_2t},\\
z(t) &=z_0\mathrm{e}^{\lambda_3t}.
\end{aligned}\right.
\end{align}
By substituting (\ref{r(t)}) into (\ref{k square}) and (\ref{tau}), we obtain
\begin{align*}
\kappa^2(t)
=& \left\{
\left[\lambda_2 \lambda_3 (\lambda_3-\lambda_2) y_0 z_0\right]^2 \cdot \mathrm{e}^{2(\lambda_2+\lambda_3)t}
+\left[\lambda_1 \lambda_3 (\lambda_1-\lambda_3) x_0 z_0\right]^2 \cdot \mathrm{e}^{2(\lambda_1+\lambda_3)t}
\right.\nonumber\\& \left.
+\left[\lambda_1 \lambda_2 (\lambda_2-\lambda_1) x_0 y_0\right]^2 \cdot \mathrm{e}^{2(\lambda_1+\lambda_2)t}
\right\}\left.\middle/
{\left\{(\lambda_1 x_0)^2 \mathrm{e}^{2\lambda_1t}
+(\lambda_2 y_0)^2 \mathrm{e}^{2\lambda_2t}
+(\lambda_3 z_0)^2 \mathrm{e}^{2\lambda_3t}\right\}^3}
,\right.
\end{align*}
and
\begin{align*}
\tau(t)
=& \left.
{\lambda_1\lambda_2\lambda_3(\lambda_2-\lambda_1)(\lambda_3-\lambda_1)(\lambda_3-\lambda_2) x_0y_0z_0 \cdot
\mathrm{e}^{(\lambda_1+\lambda_2+\lambda_3)t}}
 \middle/
\Big\{\left[\lambda_2 \lambda_3 (\lambda_3-\lambda_2) y_0z_0\right]^2 \cdot \mathrm{e}^{2(\lambda_2+\lambda_3)t}
\right. \nonumber\\
&+\left[\lambda_1 \lambda_3 (\lambda_1-\lambda_3) x_0z_0\right]^2 \cdot \mathrm{e}^{2(\lambda_1+\lambda_3)t}
+\left[\lambda_1 \lambda_2 (\lambda_2-\lambda_1) x_0y_0\right]^2 \cdot \mathrm{e}^{2(\lambda_1+\lambda_2)t}
\Big\}. \nonumber\\
\end{align*}

\subsection{Case 2}

\

In the case of
\begin{align*}
A=
\begin{pmatrix}a&b&0\\-b&a&0\\0&0&\lambda_3\end{pmatrix}\
(a, b, \lambda_3\in\mathbb{R},\ b>0),
\end{align*}
the first, second and third derivative of $r(t)$ are
\begin{align*}
\dot{r}(t)&=Ar(t)=(ax+by,~ -bx+ay,~ \lambda_3 z)^\mathrm{T},\\
\ddot{r}(t)&=A^2r(t)
=((a^2-b^2)x+2aby,~ -2abx+(a^2-b^2)y,~ \lambda_3^2 z)^\mathrm{T},\\
\dddot{r}(t)&=A^3r(t)
=(a(a^2-3b^2)x+b(3a^2-b^2)y,~ -b(3a^2-b^2)x+a(a^2-3b^2)y,~ \lambda_3^3 z)^\mathrm{T},
\end{align*}
respectively.
Noting that
\begin{align*}
\mathrm{e}^{tA}
=\begin{pmatrix}
\mathrm{e}^{at}\cos{bt} & \mathrm{e}^{at}\sin{bt} & 0
\\[0.5ex] -\mathrm{e}^{at}\sin{bt} & \mathrm{e}^{at}\cos{bt} & 0
\\[0.5ex] 0 & 0 & \mathrm{e}^{\lambda_3 t}
\end{pmatrix},
\end{align*}
we have
\begin{align*}
\left\{
\begin{aligned}
x^2(t)+y^2(t) &= (x_0^2+y_0^2)\mathrm{e}^{2at},\\
z(t) &= z_0\mathrm{e}^{\lambda_3 t}.
\end{aligned}\right.
\end{align*}
Hence the square of curvature $\kappa(t)$ of curve $r(t)$ is
\begin{align*}
\kappa^2(t)
=& \left( \frac{\left\|\dot{r}(t)\times\ddot{r}(t)\right\|}{\left\|\dot{r}(t)\right\|^3} \right)^2
= \frac{ (a^2 + b^2)(x^2 + y^2) \cdot
  \{ \lambda_3^2 [(a-\lambda_3)^2+b^2] z^2
  + b^2(a^2 + b^2)(x^2 + y^2) \} }
 { \{(a^2 + b^2)(x^2 + y^2)+(\lambda_3 z)^2 \}^3 } \\
=& \frac{ (a^2 + b^2)(x_0^2 + y_0^2) \cdot
  \{ \lambda_3^2 [(a-\lambda_3)^2+b^2]z_0^2 \cdot \mathrm{e}^{2(a+\lambda_3)t}
  + b^2(a^2 + b^2)(x_0^2 + y_0^2) \cdot \mathrm{e}^{4at} \} }
 { \{(a^2 + b^2)(x_0^2 + y_0^2) \cdot \mathrm{e}^{2at}
 +(\lambda_3 z_0)^2 \cdot \mathrm{e}^{2\lambda_3 t} \}^3 },
\end{align*}
and the torsion of curve $r(t)$ is
\begin{align*}
\tau(t)
&=\frac{\left(\dot{r}(t),\ddot{r}(t),\dddot{r}(t)\right)}
{\left\|\dot{r}(t)\times\ddot{r}(t)\right\|^2}
=\frac{ -b\lambda_3\{(a-\lambda_3)^2+b^2\} z }
{ \lambda_3^2 \{(a-\lambda_3)^2+b^2\} z^2
  + b^2(a^2 + b^2)(x^2+y^2) } \\
&=\frac{ -b\lambda_3\{(a-\lambda_3)^2+b^2\} z_0 }
{ \lambda_3^2 \{(a-\lambda_3)^2+b^2\} z_0^2 \cdot \mathrm{e}^{\lambda_3 t}
  + b^2(a^2 + b^2)(x_0^2+y_0^2) \cdot \mathrm{e}^{(2a-\lambda_3)t} }.
\end{align*}

\subsection{Case 3}

\

In the case of
\begin{align*}
A=\begin{pmatrix}\lambda_1&1&0\\0&\lambda_1&0\\0&0&\lambda_2\end{pmatrix}\
(\lambda_1, \lambda_2\in\mathbb{R}),
\end{align*}
the first, second and third derivative of $r(t)$ are
\begin{align*}
\dot{r}(t)&=Ar(t)
=(\lambda_1 x+y,~ \lambda_1 y,~ \lambda_2 z)^\mathrm{T},\\
\ddot{r}(t)&=A^2r(t)
=(\lambda_1^2 x+2\lambda_1 y,~ \lambda_1^2 y,~ \lambda_2^2 z)^\mathrm{T},\\
\dddot{r}(t)&=A^3r(t)
=(\lambda_1^3 x+3\lambda_1^2 y,~ \lambda_1^3 y,~ \lambda_2^3 z)^\mathrm{T},
\end{align*}
respectively.
Noting that
\begin{align*}
\mathrm{e}^{tA}
=\begin{pmatrix}\mathrm{e}^{\lambda_1t}&t\mathrm{e}^{\lambda_1t}&0
\\[0.5ex] 0&\mathrm{e}^{\lambda_1t}&0
\\[0.5ex] 0&0&\mathrm{e}^{\lambda_2t}
\end{pmatrix},
\end{align*}
we have
\begin{align*}
\left\{
\begin{aligned}
x(t) &=(x_0+y_0 t)\mathrm{e}^{\lambda_1t},\\
y(t) &=y_0\mathrm{e}^{\lambda_1t},\\
z(t) &=z_0\mathrm{e}^{\lambda_2t}.
\end{aligned}\right.
\end{align*}
Hence the square of curvature $\kappa(t)$ of curve $r(t)$ is
\begin{align*}
\kappa^2(t)
=& \left( \frac{\left\|\dot{r}(t)\times\ddot{r}(t)\right\|}{\left\|\dot{r}(t)\right\|^3} \right)^2
=\frac{\left\{\lambda_1 \lambda_2 (\lambda_2-\lambda_1)yz\right\}^2
+\left\{(\lambda_1^2 x - \lambda_1 \lambda_2 x +2\lambda_1 y - \lambda_2 y)\lambda_2 z\right\}^2
+\left(-\lambda_1^2 y^2\right)^2}
{\{(\lambda_1 x+y)^2+(\lambda_1 y)^2+(\lambda_2 z)^2\}^3} \nonumber\\
=&\Big\{
\big[\lambda_1 \lambda_2 (\lambda_2-\lambda_1) y_0 z_0 \big]^2 \cdot \mathrm{e}^{2(\lambda_1+\lambda_2)t}
\nonumber
\\&
+(\lambda_2 z_0)^2 \big[\lambda_1(\lambda_1 -\lambda_2 )y_0 \cdot t + (\lambda_1^2 x_0 - \lambda_1\lambda_2 x_0 + 2\lambda_1 y_0 -\lambda_2 y_0) \big] ^2 \cdot \mathrm{e}^{2(\lambda_1+\lambda_2)t}
\nonumber
\\&
+(\lambda_1^2 y_0^2)^2 \cdot \mathrm{e}^{4\lambda_1t}
\Big\}
\left.\middle/
{\Big\{ \big[ \lambda_1 y_0 \cdot t + ( \lambda_1 x_0 + y_0 ) \big]^2 \cdot \mathrm{e}^{2\lambda_1t}
+(\lambda_1 y_0)^2 \cdot \mathrm{e}^{2\lambda_1t}
+(\lambda_2 z_0)^2 \cdot \mathrm{e}^{2\lambda_2t}\Big\}^3}
\right..
\end{align*}

If $\lambda_1=\lambda_2=0$, then $\left\|\dot{r}(t)\times\ddot{r}(t)\right\|\equiv0$ and every trajectory $r(t)$ is a straight line, therefore we have $\tau(t)\equiv0$.
If $\lambda_1^2+\lambda_2^2 \neq 0$, then the torsion of curve $r(t)$ is
\begin{align*}
\tau(t)
=&\frac{\left(\dot{r}(t),\ddot{r}(t),\dddot{r}(t)\right)}
{\left\|\dot{r}(t)\times\ddot{r}(t)\right\|^2}
=\frac{ -\lambda_1^2 \lambda_2 (\lambda_1-\lambda_2)^2 y^2 z }
{\left\{\lambda_1 \lambda_2 (\lambda_2-\lambda_1)yz\right\}^2
+\left\{(\lambda_1^2 x - \lambda_1 \lambda_2 x +2\lambda_1 y - \lambda_2 y)\lambda_2 z\right\}^2
+\left(-\lambda_1^2 y^2\right)^2}  \nonumber\\
=&\left. { -\lambda_1^2 \lambda_2 (\lambda_1-\lambda_2)^2 y_0^2 z_0 \cdot \mathrm{e}^{(2\lambda_1+\lambda_2)t} }
\middle/\right.
\Big\{
\big[\lambda_1 \lambda_2 (\lambda_2-\lambda_1) y_0 z_0 \big]^2 \cdot \mathrm{e}^{2(\lambda_1+\lambda_2)t}
\nonumber
\\&
+(\lambda_2 z_0)^2 \big[\lambda_1(\lambda_1 -\lambda_2 )y_0 \cdot t + (\lambda_1^2 x_0 - \lambda_1\lambda_2 x_0 + 2\lambda_1 y_0 -\lambda_2 y_0) \big] ^2 \cdot \mathrm{e}^{2(\lambda_1+\lambda_2)t}
\nonumber
\\&
+(\lambda_1^2 y_0^2)^2 \cdot \mathrm{e}^{4\lambda_1t}
\Big\}.
\end{align*}

\subsection{Case 4}

\

In the case of
\begin{align*}
A=\begin{pmatrix}\lambda&1&0\\0&\lambda&1\\0&0&\lambda\end{pmatrix}\
(\lambda\in\mathbb{R}),
\end{align*}
the first, second and third derivative of $r(t)$ are
\begin{align*}
\dot{r}(t)&=Ar(t)=(\lambda x+y,~ \lambda y+z,~ \lambda z)^\mathrm{T},\\
\ddot{r}(t)&=A^2r(t)
=(\lambda^2 x+2\lambda y+z,~ \lambda^2 y+2\lambda z,~ \lambda^2 z)^\mathrm{T},\\
\dddot{r}(t)&=A^3r(t)
=(\lambda^3 x+3\lambda^2 y+3\lambda z,~ \lambda^3 y+3\lambda^2 z,~ \lambda^3 z)^\mathrm{T},
\end{align*}
respectively.
Noting that
\begin{align*}
\mathrm{e}^{tA}
=\mathrm{e}^{\lambda t}
\begin{pmatrix}
1 & t & \frac12 t^2
\\[0.5ex]0 & 1 & t
\\[0.5ex]0 & 0 & 1
\end{pmatrix},
\end{align*}
we have
\begin{align}\label{r(t) 4}
\left\{
\begin{aligned}
x(t)&=\left( x_0 + y_0 t + \frac12 z_0 t^2 \right) \mathrm{e}^{\lambda t},\\
y(t)&=(y_0 + z_0 t) \mathrm{e}^{\lambda t},\\[1ex]
z(t)&=z_0\mathrm{e}^{\lambda t}.
\end{aligned}\right.
\end{align}
By substituting (\ref{r(t) 4}) into
$\left\|\dot{r}(t)\times\ddot{r}(t)\right\|^2$,
$\left\|\dot{r}(t)\right\|^6$ and
$\left(\dot{r}(t),\ddot{r}(t),\dddot{r}(t)\right)$,
we have
\begin{align*}
\left\|\dot{r}(t)\times\ddot{r}(t)\right\|^2 &= f(t) \cdot \mathrm{e}^{4\lambda t}, \\
\left\|\dot{r}(t)\right\|^6 &= g(t) \cdot \mathrm{e}^{6\lambda t},
\end{align*}
and
\begin{align*}
\left(\dot{r}(t),\ddot{r}(t),\dddot{r}(t)\right) = -\lambda^3 z_0^3 \cdot \mathrm{e}^{3\lambda t},
\end{align*}
where $f(t)$ and $g(t)$ are polynomials in $t$.
Thus, we obtain
\begin{align*}
\kappa^2(t)
=\left( \frac{\left\|\dot{r}(t)\times\ddot{r}(t)\right\|}{\left\|\dot{r}(t)\right\|^3} \right)^2
=\frac{ f(t) }{ g(t) \cdot \mathrm{e}^{2\lambda t} },
\end{align*}
and
\begin{align*}
\tau (t)
=\frac{\left(\dot{r}(t),\ddot{r}(t),\dddot{r}(t)\right)}{\left\|\dot{r}(t)\times\ddot{r}(t)\right\|^2}
=\frac{ -\lambda^3 z_0^3 }{ f(t) \cdot \mathrm{e}^{\lambda t} }.
\end{align*}
If $\lambda\neq0$, then
$f(t)= \left( \frac14 \lambda^4 z_0^4 \right) \cdot t^4 + \sum_{i=0}^3 a_i t^i$ is a quartic polynomial in $t$,
and $g(t)= \left( \frac{1}{64} \lambda^6 z_0^6 \right) \cdot t^{12} + \sum_{i=0}^{11} b_i t^i$ is a polynomial of degree 12 in $t$;
if $\lambda=0$, then $f(t)=z_0^4$ is a constant, and $g(t)=z_0^6 t^6 + \sum_{i=0}^{5} b_i t^i$ is a polynomial of degree $6$ in $t$.

\end{document}